\documentclass{article}
\usepackage{amsmath}
\usepackage{paralist}
\usepackage{graphics}
\usepackage{amsfonts,amssymb}
\usepackage{color}
\usepackage{amsthm}
\usepackage{mathrsfs}
\usepackage{cite}
\date{}

\newtheorem{Theorem}{Theorem}[section]
\newtheorem{Lemma}{Lemma}[section]
\newtheorem{Proposition}[Lemma]{Proposition}
\newtheorem{Remark}{Remark}[section]

\def\R{{\mathbb{R}}}

\title{Global Well-posedness for 2D non-resistive MHD equations in half-space  }

\author{  Zhaoyun Zhang$^1$, \  \ Xiaopeng Zhao$^2$\thanks{Corresponding author: zhaoxiaopeng@mail.neu.edu.cn 
 \ \    }
    \\
 {\small
  1.  School of Mathematical Sciences,
     Shanghai Jiao Tong University, Shanghai 200201,  China}\\
  {\small
  2. College of Sciences, Northeastern University, Shenyang 110819,    China}
 }

\begin{document}

\maketitle
\bibliographystyle{plain}
\baselineskip=16pt

 \begin{abstract}
 This paper focuses
on the initial boundary value problem of two-dimensional non-resistive MHD equations in a half space. We prove that the MHD equations have a unique global strong solution around the equilibrium  state $(0,\bf{e_1})$ for Dirichlet boundary condition of velocity and modified Neumann boundary condition of magnetic.\\

{\small{\it MSC}: 35A01; 35Q30; 76D05}\\

 {\small{\it Key words}: Non-resistive MHD,
  \   Global regularity, \ Asymptotic estimates, \ Half space }
\end{abstract}

\bigskip

\section{Introduction}
In this paper, we consider the initial boundary value problem of 
2D incompressible non-resistive magnetohydrodynamic equations (GMHD)
\begin{equation}\label{MHD}
 \left\{
\begin{array}{l}
u_t + u\cdot \nabla u -\Delta u = -\nabla \pi  + B\cdot\nabla B, ~~x\in \R^2_+, ~t>0,\\
B_t + u\cdot \nabla B  =   B\cdot\nabla u, ~~x\in \R^2_+,~t>0,\\
\nabla\cdot u =0, \quad \nabla\cdot B =0, ~~x\in \R^2_+,~t>0,
 \end{array} \right.
 \end{equation}
where $u=u(x,t)=(u^1(x,t),u^2(x,t))$, $B=B(x,t)=(B^1(x,t),B^2(x,t))$ and $\pi=\pi(x,t)$
 denote  the velocity field,  the magnetic field and the pressure, respectively. In addition, $\R^2_+:=\{x=(x_1,x_2)\in \R^2, x_1\in (-\infty, +\infty), x_2\in (0,+\infty)\}$, thus, the outer normal vector of $\partial \R^2_+$ is ${\bf n}:=(0,-1)$.
 
 It is well-known that MHD equations model the dynamics of the velocity and magnetic fields, which plays
an important role in describing electrically conducting fluids, such as plasmas, liquid metals. When the magnetic diffusion is included, it is well-known that the 2D MHD equations have the global smooth solution\cite{1972DL,1983ST}. Moreover, by following Kato's arguments \cite{1984K} and applying the $L^p$-$L^q$ estimates of Stokes semigroup established in \cite{1987U}, there exists a unique strong solution in half-space $\R^n_+$ when initial data are small (small condition is unnecessary if $n=2$), e.g. \cite{2005HX,2012HH}.
 But without magnetic diffusion, the question of whether smooth solution of the 2D MHD equations develops singularity in finite time has been a long-standing open problem. The local existence result of smooth solutions for non-resistive MHD equations was established in \cite{2006JN,2014FMRR,2017FMRR}. In addition, more recently Ren et al.\cite{2014RWXZ} and Lin et al. \cite{2015LXZ} have established the existence of global-in-time solutions for initial data are small and sufficiently close to certain equilibrium solutions in whole space, respectively. Further, Ren et al. \cite{2016RXZ} have established similar results in a strip domain with non-slip boundary condition and Navier slip boundary condition on the velocity. 
 
 In this paper, we are concerned with the non-slip boundary condition on the velocity $u$, i.e.
   \begin{equation}\label{1.2}
 	u(x,t)|_{x\in\partial \R^2_+}=0,
 \end{equation}
 as well as that the boundary condition for the magnetic $B$, i.e.
 \begin{equation}\label{1.3}
 	\nabla^T B(x,t)\cdot {\bf n}|_{x\in \partial \R^2_+}=0,
 \end{equation}
where $\nabla^T=(-\partial_2,\partial_1)$. More precisely,
we give the detail of the boundary condition for the magnetic $B$. 

In fact, $\nabla^T B(x,t)\cdot {\bf n}=0$ on $\partial \R^2_+$, which deduces that 
 $\partial_1 B=0$ on $\partial \R^2_+$, 
 i.e. $\partial_1B^1=\partial_1B^2=0$ on $\partial \R^2_+$.
 Due to $\nabla\cdot B=0$, $\partial_2B^2=0$. Since $\partial_1B^2=0$ and $B^2(x_1,x_2,t)=0$ as $|x|\rightarrow\infty$, therefore $B^2=0$ on $\partial \R^2_+$. The questions about well-posedness become more subtle and more difficult when physical boundaries are present. However, it seems that this question can be handled when boundary conditions (\ref{1.2}) and (\ref{1.3}) are given.
 
 We will investigate small perturbations of the system (\ref{MHD}) around the equilibrium state $(0,\bf{e}_1)$ in half-space $\R^2_+$. Hence, setting $b=B-\bf{e}_1$ and reformulating (\ref{MHD}) as follows
 \begin{equation}\label{PMHD}
 	\left\{
 	\begin{array}{l}
 		u_t + u\cdot \nabla u -\Delta u -\partial_1b= -\nabla \pi  + b\cdot\nabla b, ~~x\in \R^2_+,t>0 \\
 		b_t + u\cdot \nabla b -\partial_1u =   b\cdot\nabla u, ~~x\in\R^2_+,t>0\\
 		\nabla\cdot u = \nabla\cdot b =0, ~~x\in \R^2_+,t>0 \\
 		u=0,~~\nabla^T b\cdot {\bf n}=0, ~~x\in \partial\R^2_+,t>0\\
 		u(x,0)=u_0(x), ~~b(x,0)=b_0(x),~~x\in \R^2_+.
 	\end{array} \right.
 \end{equation}

There are two main difficulties in the existence of strong solutions. The first one arises from the partial dissipation of magnetic $b$. In order to control the nonlinearities, it needs the dissipation of velocity $u$ and magnetic $b$. It is known that the dissipation of $u$ is easy to obtained, whereas the dissipation of $b$ is more subtle without the magnetic diffusion. Fortunately, this difficulty can be overcome by a delicate potential estimates. Motivated by\cite{2016RXZ,2016Z}, through replacing $\Delta\partial_1 u$ by $\partial_1\partial_t u-\partial_1^2b+\partial_1\nabla \pi-\partial_1(b\cdot\nabla b-u\cdot\nabla u) $, inner product with $\mathcal{A}b$, combining with that ${\bf P}$ commutates with $\partial_1$ in half-space,  it follows that the dissipation for $b$ is of form $\|\partial_1\nabla b\|_{L^2}$ rather than $\|\nabla^2 b\|_{L^2}$. The other difficulty  results from the nonlinearity in half-space, i.e. ${\bf P}$ cannot communate with $\nabla$. This needs a series of careful energy estimates to deal with that.
 %\cite{2010H} $${\bf P}(u\cdot\nabla u)=u\cdot\nabla u+\sum_{i,j=1}^n\nabla \mathcal{N}\partial_i\partial_j(u_iu_j) $$ 

As for the long time behaviour of solutions, we are only able to obtain the time asymptotic estimates without decay rates. Without the magnetic diffusion and boundary condition present, it leads to some trouble in dealing with decay rates, which will be left to our future's work.  

The rest of this paper is organized as follows. In section 2, we give some basic interpolation inequalities, properties about the Helmholtz projection ${\bf P}$ and our main results.  The global well-posedness of solutions is established in section 3. Finally, we prove the time asmyptotic estimates in section 4.

\vspace{6mm}

\setcounter{equation}{0}
\section{ Preliminaries and Main Result}

\begin{Lemma}
	(1-1)\cite{2011B,2014P} Let $f\in H^1(\R^2_+)$. Then, for any $2<p<\infty$ and there is $C=C(p)$,
	\begin{equation}\label{2.1}
			\|f\|_{L^p(\R^2_+)}\leq C\|u\|_{L^2(\R^2_+)}^{\frac2p}\|\nabla f\|_{L^2(\R^2_+)}^{\frac{p-2}{p}}.
	\end{equation}
(1-2)\cite{2011B} For $m\geq 1$ is an integer and $p\in [1,+\infty)$, we have $$W^{m,p}(\R^n_+)\subset L^\infty(\R^n_+), ~~\frac1p-
\frac mn<0,$$
and
 $$W^{m,p}(\R^n_+)\subset L^q(\R^n_+), ~\forall q\in [p,+\infty), ~\frac{1}{p}-\frac{m}{n}=0.$$
(2)\cite{1951H}
For any $f=(f^1,f^2)\in H^2(\R^2_+)$ with $\nabla\cdot f=0$ in $\R^2_+$ and $f^2=0$ on $\partial \R^2_+$. There exists a positive constant C such that
\begin{equation}\label{2.2}
  \|f^2\|_{L^\infty}\leq C\|\partial_1 f\|_{H^1};
\end{equation}
\begin{equation}\label{2.3}
	\|f\|_{L^\infty}\leq C\|f\|_{H^1}^{\frac12}\|\partial_1f\|_{H^1}^{\frac12};
\end{equation}
\begin{equation}\label{2.4}
	\|f\|_{W^{1,4}}\leq C\|f\|_{H^2}^{\frac34}\|\partial_1f\|_{H^1}^{\frac14}.
\end{equation}
\end{Lemma}
Denote by $\mathbf{P}$ the Helmholtz projection and define $\mathcal{A} :={\bf P}(-\Delta)$. Then, the following results hold true. 
\begin{Lemma}\cite{1988BM}
		For any $1<p<\infty$ and $n\geq2$,\\
~~	(1) $\|\mathbf{P}f\|_{L^p(\R^n_+)}\leq C\|f\|_{L^p(\R^n_+)}$.\\
~~(2)	$\|\nabla^2 f(t)\|_{L^p(\R^n_+)}\leq C\|\mathcal{A} f(t)\|_{L^p(\R^n_+)}$.
\end{Lemma}
\begin{Lemma}
	 In $\R^2_+:=\{x_1\in \R, x_2\in (0,+\infty)\}$,\\
~~	(1)
	$	\partial_1{\bf P}={\bf P}\partial_1({\bf P} \partial_2\neq \partial_2{\bf P})$;\\
~~(2) ${\bf P}^2={\bf P}$.
\end{Lemma}
\begin{Lemma}\cite{2011H,2012H}
	Let $1\leq k,m\leq n$, then for all $u\in C_{0,\sigma}^\infty(\R^n_+)$
	\begin{equation*}
		\|\sum_{i,j=1}^n\partial_k\partial_m\mathcal{N}\partial_i\partial_j(u^iu^j)\|_{L^q(\R^n_+)}\leq C(\|u\|_{L^{2q}(\R^n_+)}^2+\|\nabla u\|_{L^{2q_1}(\R^n_+)}\|\nabla^2u\|_{L^{2q_2}(\R^n_+)}),
	\end{equation*}
for any $1\leq q\leq \infty$ and $\frac{1}{q_1}+\frac{1}{q_2}=\frac{1}{q}$, $1\leq q_1,q_2\leq +\infty$.
\end{Lemma}

 Our main result is the following:
\begin{Theorem}
	Let $(u_0,b_0)\in H^2(\R^2_+)$ satisfy $\nabla\cdot u_0=\nabla\cdot b_0=0$ in $\R^2_+$. If the initial data are small and $b^1(0,\cdot)|_{\partial\R^2_+}=0$,
	% and $$\|(u_0,b_0)\|_{H^2}^2+\|u_t(0)\|_{L^2}^2\leq \epsilon^2$$
%	for some suitable small $\epsilon>0$, 
	then the MHD system (\ref{PMHD}) admits a unique global solution $(u,b)\in C([0,+\infty);H^2(\R^2_+)$ and there exists some constant $C_0>0$ such that 
	\begin{eqnarray*}
		&&\|(u,b)\|_{H^2}^2+\|u_t\|_{L^2}^2+\|\nabla \pi\|_{L^2}^2+\int_0^t\|\nabla u\|_{H^2}^2+\|(\partial_1b,u_t)\|_{H^1}^2+\|b_t\|_{L^2}^2ds\\
		&\leq& C_0(\|(u_0,b_0)\|_{H^2}^2+\|u_t(0)\|_{L^2}^2)
	\end{eqnarray*}
for all $t>0$. In addition, for $2<p<+\infty$
\begin{equation*}
	\|(u,b)(\cdot,t)\|_{L^p}\rightarrow 0, ~as~t\rightarrow+\infty.
\end{equation*}
\end{Theorem}
\begin{Remark}
	 The initial data $b^1_0(\cdot)|_{\partial \R^2_+}=0$ is essential in our proof. This because that it can be propagated for any time by the equation $(\ref{PMHD})_2$ if $b^1_0=0$ on $\partial \R^2_+$. Indeed,  since $(u^1,u^2)=0$, $b^2=0$ and $\partial_1u^1=0$ on $\partial\R^2_+$, it holds that
	\begin{equation*}
		b^1_t+\sum u^j\partial_j b^1-\partial_1u^1=b^1\partial_1u^1+b^2\partial_2u^2, ~~x\in \partial \R^2_+,
	\end{equation*}
	\begin{equation*}
		\Rightarrow	b^1_t=0,~~ x\in \partial \R^2_+,
	\end{equation*}
\begin{equation*}
	\Rightarrow	b^1(t,x)=b^1(0,x),~~ x\in \partial \R^2_+,t>0.
\end{equation*}
\end{Remark}

\vspace{6mm}

\setcounter{equation}{0}
\section{ Global well-posedness}
The following local well-posedness can be proved by using the standard energy method. So we omit its proof.
\begin{Theorem}
	Assume that the initial data $(u_0,b_0)$ satisfies $(u_0,b_0)\in H^2 $. Then there exists $T>0$ such that the MHD system (\ref{PMHD}) has a unique solution $(u,b)$
	on $[0,T]$ satisfying 
	$$ (u,b)\in C([0,T]; H^2(\R^2_+)).$$
\end{Theorem}
Now let us introduce the following energy
$$\mathcal{E}(t):=\|(u,b)\|_{H^2}^2+\|(u_t,\nabla \pi)\|_{L^2}^2,$$
and the dissipated energy
$$\mathcal{F}(t):=\int_0^t\|\nabla u\|_{H^2}^2+\|(\partial_1b,u_t)\|_{H^1}^2+\|b_t\|_{L^2}^2d\tau.$$
\begin{Proposition}
	Assume that the initial data of the system (\ref{PMHD}) are small enough, then for some $C>0$ and $0<T\ll 1$, it holds that 
	$$\mathcal{E}(t)+\mathcal{F}(s)\leq C\mathcal{E}(0)$$
	for any $t\in (0,T]$.
\end{Proposition}
\begin{proof}
	{\bf Step 1.} $L^2$ estimate of $(u,b)$.
	
	We take the $L^2$ inner product of equations $(\ref{PMHD})_1$ and $(\ref{PMHD})_2$ with u and b, respectively. 
	\begin{equation}\label{4.1}
		\frac12\frac{d}{dt}(\|u\|_{L^2}^2+\|b\|_{L^2}^2)+\|\nabla u\|_{L^2}^2=0
	\end{equation}
	for any $t\in (0,T]$.\\
{\bf Step 2.} $\dot{H}^1$ estimate of $(u,b)$.	

 Define  
$$\mathcal{A}u:=-\mathbf{P}\Delta u$$
for any $u\in L^2_{\sigma}\cap H^1_0\cap H^2$, where $\mathbf{P}$ is the Helmholtz projection.

We take the $L^2$ product of equation $(\ref{PMHD})_1$ with $\mathcal{A}u$ and apply $\nabla$ to equation $(\ref{PMHD})_2$ and take the $L^2$ inner product with $\nabla b$ to obtain
	\begin{eqnarray}\label{4.2}
		&&\frac12\frac{d}{dt}(\|\nabla u\|_{L^2}^2+\|\nabla b\|_{L^2}^2)+\|\mathcal{A}u\|_{L^2}^2\nonumber\\
		&=&\Big(\langle b\cdot\nabla b,\mathcal{A}u\rangle-\langle u\cdot \nabla u,\mathcal{A}u\rangle \Big)+\Big(\langle \nabla(b\cdot\nabla u),\nabla b\rangle-\langle \nabla(u\cdot\nabla b),\nabla b\rangle\Big)\nonumber\\
		&&+\langle \partial_1b,\mathcal{A}u\rangle+\langle \nabla\partial_1u,\nabla b\rangle\nonumber\\
		&:=&I_1+I_2+I_3+I_4.
	\end{eqnarray}
First, the pressure is disappeared since  
\begin{eqnarray*}
\int_{\R^2_+}\nabla \pi \mathcal{A}udx=-\int_{\R^2_+}\mathbf{P}(\nabla \pi)\Delta udx=0.
\end{eqnarray*}	
We will take some estimates to $I_{i}(i=1,2,3,4)$.\\
	For $I_1$, we rewrite it as
	$$I_1=\langle b^1\partial_1b,\mathcal{A}u\rangle+\langle b^2\partial_2b,\mathcal{A}u\rangle-\langle u\cdot \nabla u,\mathcal{A}u\rangle.$$
Using Lemma 2.1 and Lemma 2.2 to obtain
\begin{eqnarray}\label{4.3}
	I_1&\leq& \|b^1\|_{L^\infty}\|\partial_1b\|_{L^2}\|\Delta u\|_{L^2}+\|b^2\|_{L^\infty}\|\partial_2b\|_{L^2}\|\Delta u\|_{L^2}+\|u\|_{L^\infty}\|\nabla u\|_{L^2}\|\Delta u\|_{L^2}\nonumber\\
	&\leq&\|b\|_{H^2}\|\partial_1b\|_{L^2}\|\Delta u\|_{L^2}+\|\partial_1b\|_{H^1}\|\nabla b\|_{L^2}\|\Delta u\|_{L^2}+\|u\|_{H^2}\|\nabla u\|_{H^1}^2\nonumber\\
	&\leq& (\|u\|_{H^2}+\|b\|_{H^2})(\|\nabla u\|_{H^1}^2+\|\partial_1b\|_{H^1}^2).
\end{eqnarray}	
	For $I_2$, we split it into four terms:
	\begin{eqnarray}\label{4.4}
		I_2&=&\Big(\langle \nabla b^1,\nabla (b\cdot \nabla u^1)\rangle-\langle \nabla b^1,\nabla u\cdot \nabla b^1\rangle\Big )+\Big(\langle \nabla b^2,\nabla(b\cdot \nabla u^2)\rangle-\langle \nabla b^2,\nabla u\cdot\nabla b^2\rangle\Big)\nonumber\\
		&:=&I_{21}+I_{22},
	\end{eqnarray}
where integrating by parts yields $\langle u\cdot\nabla \nabla b,\nabla b\rangle=0$.
We rewrite $I_{21}$ as
\begin{eqnarray*}
	I_{21}&=&\Big(\langle \partial_1b^1,\partial_1(b\cdot\nabla u^1)\rangle-\langle \partial_1b^1,\partial_1u\cdot\nabla b^1\rangle\Big)+\langle \partial_2b^1,\partial_2(b\cdot\nabla u^1)\rangle-\langle \partial_2b^1,\partial_2u\cdot \nabla b^1\rangle\\
	&=&\Big( \langle \partial_1b^1,\partial_1b\cdot\nabla u^1\rangle+\langle \partial_1b^1,b\cdot\nabla \partial_1u^1\rangle-\langle \partial_1b^1,\partial_1u\cdot\nabla b^1\rangle\Big)\\
	&&+\langle \partial_2b^1,\partial_2b^1\partial_1 u^1\rangle+\langle \partial_2b^1,b^1\partial_1\partial_2 u^1)\rangle-\langle \partial_2b^1,\partial_2u^1 \partial_1 b^1\rangle\\
	&&+\langle \partial_2b^1,\partial_2b^2\partial_2 u^1\rangle+\langle \partial_2b^1,b^2\partial_2\partial_2 u^1)\rangle-\langle \partial_2b^1,\partial_2u^2 \partial_2 b^1\rangle\\
	&:=&\Big(I_{211}+I_{212}+I_{213}\Big)+\cdots+I_{219}.
\end{eqnarray*}	
It follows from the H$\ddot{o}$lder's inequality and Lemma 2.1 (1-2) that
\begin{eqnarray*}
	\Big(I_{211}+I_{212}+I_{213}\Big)&\leq& \|\partial_1b^1\|_{L^2}\|\nabla b\|_{L^4}\|\nabla u\|_{L^4}+\|\partial_1b^1\|_{L^2}\|b\|_{L^\infty}\|\nabla \partial_1 u^1\|_{L^2}\\
	&\leq& \|\partial_1 b\|_{L^2}\|\nabla b\|_{H^1}\|\nabla u\|_{H^1}+\|\partial_1 b\|_{L^2}\|b\|_{H^2}\|\nabla u\|_{H^1}\\
	&\leq&\| b\|_{H^2}(\|\partial_1b\|_{L^2}^2+\| \nabla u\|_{H^1}^2).
\end{eqnarray*}	 
The higher regularity of $\partial_1b$ comes from the following term
	\begin{eqnarray*}
		I_{215}&=&-\int \partial_1 \partial_2b^1b^1\partial_2 u^1dx-\int \partial_2b^1\partial_1b^1\partial_2u^1dx \\
		&\leq& \|\partial_2u^1\|_{L^2}\|\partial_1\partial_2b^1\|_{L^2}\|b^1\|_{L^\infty}+\|\partial_2 u^1\|_{L^2}\|\partial_2b^1\|_{L^4}\|\partial_1b^1\|_{L^4}  \\
		&\leq& (\|\nabla u\|_{L^2}^2+\|\partial_1 b\|_{H^1}^2)\|b\|_{H^2};
	\end{eqnarray*}
due to $\partial_2b^2=-\partial_1b^1$,
	\begin{eqnarray*}
		I_{217}&=& \int \partial_2b^1\partial_2b^2\partial_2u^1dx=-\int\partial_2b^1\partial_1b^1\partial_2u^1dx\\
		&\leq& \|\partial_2b^1\|_{L^4}\|\partial_1b^1\|_{L^2}\|\partial_2u^1\|_{L^4} \leq\| b\|_{H^2}(\|\partial_1b\|_{L^2}^2+\|\nabla u\|_{H^1}^2).
	\end{eqnarray*}
The higher regularity of $\partial_1b$ also comes from that 
\begin{eqnarray*}
	I_{218}&\leq&\|\nabla b\|_{L^2}\|b^2\|_{L^\infty}\|\nabla^2u\|_{L^2}\leq \|\nabla b\|_{L^2}\|\partial_1b\|_{H^1}\|\nabla^2u\|_{L^2}\\
	&\leq& \|b\|_{H^1}(\|\partial_1b\|_{H^1}^2+\|\nabla u\|_{H^1}^2),
\end{eqnarray*}	
where we have used the Lemma 2.1 (2).\\
Integrating by parts yields
\begin{eqnarray*}
		I_{214}+I_{219}&=&2\int (\partial_2b^1)^2\partial_1u^1dx=2\mathcal{K}.
	%=-4\int \partial_2b^1\partial_1\partial_2b^1 u^1dx\\
	%	& =&4\int \partial_2(\partial_2b^1u^1)\partial_1b^1dx+4\int_{\partial \R^2_+}\partial_2b^1u^1\partial_1b^1dS\\
	%	&=&4\int \partial_2(\partial_2b^1u^1)\partial_1b^1dx\\
	%	&=&4\int \partial_2^2b^1u^1\partial_1b^1+\partial_2u^1\partial_2b^1\partial_1b^1dx\\
	%	&\leq& \|\partial_1b^1\|_{L^4}\|u^1\|_{L^4}\|\partial_2
	%	^2b^1\|_{L^2}+\|\partial_1b^1\|_{L^2}\|\partial_2u^1\|_{L^4}\|\partial_2b^1\|_{L^4}\\
	%	&\leq&\|b\|_{H^2}(\|u\|_{H^1}^2+\|\partial_1b\|_{H^1}^2)+\|\nabla b\|_{H^1}(\|\nabla u\|_{H^1}^2+\|\partial_1b\|_{L^2}^2)\\
	%	&\leq &C\|b\|_{H^2}(\|u\|_{H^2}^2+\|\partial_1b\|_{L^2}^2).
\end{eqnarray*}
Next, we establish the estimate of the term $\mathcal{K}$. Using equation $(\ref{PMHD})_2$, we obtain
\begin{eqnarray*}
	\mathcal{K}&=&\langle (\partial_2b^1)^2,\partial_t b^1\rangle+\langle (\partial_2b^1)^2, u\cdot\nabla b^1-b\cdot\nabla u^1\rangle\\
	&=&\frac{d}{dt}\langle (\partial_2b^1)^2,b^1 \rangle-2\langle \partial_2\partial_tb^1,b^1\partial_2b^1\rangle+\langle (\partial_2b^1)^2,u\cdot\nabla b^1-b\cdot\nabla u^1\rangle \\
	&=&\frac{d}{dt}\langle (\partial_2b^1)^2,b^1 \rangle-2\langle \partial_2(\partial_1u^1-u\cdot\nabla b^1+b\cdot\nabla u^1),b^1\partial_2b^1\rangle\\
	&&+\langle  (\partial_2b^1)^2,u\cdot\nabla b^1-b\cdot\nabla u^1 \rangle.
\end{eqnarray*}
So,
\begin{eqnarray*}
	&&-\frac{d}{dt}\langle (\partial_2b^1)^2,b^1 \rangle+\mathcal{K}\\
	&=&-2\langle \partial_1\partial_2u^1,b^1\partial_2b^1\rangle+2\langle \partial_2u\cdot \nabla b^1,b^1\partial_2b^1\rangle+2\langle u\cdot\nabla \partial_2b^1,b^1\partial_2b^1\rangle\\
	&&-2\langle \partial_2b\cdot\nabla u^1,b^1\partial_2b^1\rangle-2\langle b\cdot\nabla \partial_2u^1,b^1\partial_2b^1\rangle+\langle  (\partial_2b^1)^2,u\cdot\nabla b^1 \rangle\\
	&&-\langle  (\partial_2b^1)^2,b\cdot\nabla u^1 \rangle\\
	&:=&\sum_{i=1}^{7}A_i.
\end{eqnarray*}
It follows that 
\begin{eqnarray*}
	A_3+A_6=\langle u\cdot\nabla (\partial_2b^1)^2,b^1\rangle+\langle  (\partial_2b^1)^2,u\cdot\nabla b^1 \rangle=0;
\end{eqnarray*}
integrating by parts, together with $\partial_1 b|_{\partial \R^2_+}=0$ implies
\begin{eqnarray*}
	A_1&=&-2\langle \partial_2^2u^1, b^1\partial_1b^1\rangle-2\int_{\partial\R^2_+} \partial_2u^1b^1\partial_1b^1dS=-2\langle \partial_2^2u^1, b^1\partial_1b^1\rangle\\
	&\leq& \|\nabla^2u\|_{L^2}\|b\|_{L^\infty}\|\partial_1b^1\|_{L^2}
	\leq\|\nabla u\|_{H^1}\|b\|_{H^2}\|\partial_1b\|_{L^2}\\
	&\leq&\|b\|_{H^2}(\|\nabla u\|_{H^1}^2+\|\partial_1b\|_{L^2}^2).
\end{eqnarray*}
 It follows from (\ref{2.2})-(\ref{2.4}) of Lemma 2.1 that
\begin{eqnarray*}
	A_2+A_4&\leq&\|\nabla u\|_{L^2}\|\nabla b\|_{L^4}\|b^1\|_{L^\infty}\|\nabla b\|_{L^4}\\
	&\leq&\|\nabla u\|_{L^2}\|b\|_{H^2}^{\frac34}\|\partial_1b\|_{H^1}^{\frac14}\|b\|_{H^1}^{\frac12}\|\partial_1b\|_{H^1}^{\frac12}\|b\|_{H^2}^{\frac34}\|\partial_1b\|_{H^1}^{\frac14}\\
	&\leq&\|b\|_{H^2}^2(\|\nabla u\|_{L^2}^2+\|\partial_1b\|_{H^1}^2);
\end{eqnarray*}
\begin{eqnarray*}
	A_5&\leq& \|b\|_{L^\infty}^2\|\nabla^2u\|_{L^2}\|\nabla b\|_{L^2}\leq \|b\|_{H^1}\|\partial_1b\|_{H^1}\|\nabla^2u\|_{L^2}\|\nabla b\|_{L^2}\\
	&\leq&\|b\|_{H^1}^2(\|\nabla u\|_{H^1}^2+\|\partial_1b\|_{H^1}^2);
\end{eqnarray*}
\begin{eqnarray*}
	A_7&\leq& \|b\|_{L^\infty}\|\nabla u\|_{L^2}\|(\partial_2b^1)^2\|_{L^2}\leq \|b\|_{H^1}^{\frac12}\|\partial_1b\|_{H^1}^{\frac12}\|\nabla u\|_{L^2}\|\nabla b\|_{L^4}^2\\
	&\leq&\|b\|_{H^1}^{\frac12}\|\partial_1b\|_{H^1}^{\frac12}\|\nabla u\|_{L^2}\|b\|_{H^2}^{\frac32}\|\partial_1b\|_{H^1}^{\frac12}\\
	&\leq&\|b\|_{H^2}^2(\|\nabla u\|_{L^2}^2+\|\partial_1b\|_{H^1}^2).
\end{eqnarray*}
Hence, we combine all these estimates and derive that 
\begin{equation}\label{4.5}
	I_{21}\leq C(\|b\|_{H^2}+\|b\|_{H^2}^2)(\|\nabla u\|_{H^1}^2+\|\partial_1b\|_{H^1}^2).
\end{equation}
	For $I_{22}$, with the aid of $\partial_2b^2=-\partial_1b^1$ and the interpolation, we have 
	\begin{eqnarray}\label{4.6}
	I_{22}&=&\langle \partial_1b^2,\partial_1(b\cdot\nabla u^2)\rangle+\langle \partial_2b^2,\partial_2(b\cdot\nabla u^2)\rangle-\langle \partial_1b^2,\partial_1u\cdot\nabla b^2\rangle-\langle \partial_2b^2,\partial_2u\cdot\nabla b^2\rangle\nonumber\\
	&\leq&\|\partial_1b\|_{L^2}(\|\partial_1(b\cdot\nabla u^2)\|_{L^2}+\|\partial_2(b\cdot\nabla u^2)\|_{L^2}+\|\partial_1u\cdot\nabla b^2\|_{L^2}+\|\partial_2u\cdot\nabla b^2\|_{L^2})\nonumber\\
	&\leq& \|\partial_1b\|_{L^2}(\|\nabla b\|_{L^4}\|\nabla u\|_{L^4}+\|b\|_{L^\infty}\|\nabla^2u\|_{L^2})\nonumber\\
	&\leq&\|b\|_{H^2}(\| \nabla u\|_{H^1}^2+\|\partial_1b\|_{L^2}^2).
	\end{eqnarray}
For $I_3+I_4$, thanks to Lemma 2.3 (1) and $\partial_1 b|_{\partial\R^2_+}=0$, there holds
\begin{eqnarray}\label{4.7}
	I_3+I_4&=&\int_{\R^2_+} \partial_1b \mathcal{A} udx+\int_{\R^2_+} \nabla \partial_1 u\nabla bdx\nonumber\\
	&=&-\int_{\R^2_+} \partial_1b \Delta udx+\int_{\R^2_+} \nabla \partial_1 u\nabla bdx\nonumber\\ 
	&=&\int_{\R^2_+}\nabla \partial_1b\nabla udx+\int_{\partial \R^2_+}\nabla u\partial_1b\cdot ndS-\int_{\R^2_+}\nabla u\nabla \partial_1bdx\nonumber\\
	&=&0.
\end{eqnarray}
Therefore, by substituting (\ref{4.3})-(\ref{4.7}) into (\ref{4.2}), we have
	\begin{eqnarray}\label{4.8}
	\frac12\frac{d}{dt}(\|\nabla u\|_{L^2}^2+\|\nabla b\|_{L^2}^2)+\|\mathcal{A}u\|_{L^2}^2
	\leq C(\|u\|_{H^2}+\|b\|_{H^2}+\|b\|_{H^2}^2)(\| \nabla u\|_{H^1}^2+\|\partial_1b\|_{H^1}^2).
\end{eqnarray}

{\bf Step 3.} Dissipation estimate of $\partial_1b$.

Taking the $L^2$ product of equations $(\ref{PMHD})_1$ and $(\ref{PMHD})_2$ with $-\partial_1b$ and $\partial_1u$, respectively, then integrating by parts, we obtain
\begin{eqnarray}\label{4.9}
	&&\frac{d}{dt}\langle b,\partial_1u\rangle+\|\partial_1b\|_{L^2}^2-\|\partial_1u\|_{L^2}^2\nonumber\\
	&=&\Big(\langle \partial_1b,u\cdot\nabla u\rangle-\langle \partial_1b,b\cdot\nabla b\rangle\Big)-\Big(\langle \partial_1u,u\cdot \nabla b\rangle+\langle \partial_1u,b\cdot\nabla u\rangle\Big)-\langle \Delta u,\partial_1b\rangle\nonumber\\
	&:=&J_1+J_2+J_3,
\end{eqnarray}
where the first term is derived from $$\langle u_t,-\partial_1b\rangle+\langle b_t,\partial_1u\rangle=\frac{d}{dt}\langle b,\partial_1u\rangle,$$
 since $\nabla\cdot b=0$ and $\partial_1b^2|_{\partial \R^2_+}=0$, we eliminate the pressure term by the following equality
\begin{eqnarray*}
	\int \nabla \pi \partial_1bdx=-\int \pi\partial_1\nabla\cdot bdx+\int_{\partial \R^2_+} \pi\partial_1b\cdot ndS=-\int_{\partial \R^2_+} \pi\partial_1b^2dS=0.
\end{eqnarray*}
For the term $J_1$,
\begin{eqnarray*}
	J_1&=&\langle \partial_1b,u\cdot\nabla u\rangle-\langle \partial_1b,b^1\partial_1b\rangle-\langle \partial_1b,b^2\partial_2b\rangle\\
	&\leq&\|\partial_1b\|_{L^2}\|u\|_{L^\infty}\|\nabla u\|_{L^2}+\|\partial_1b\|_{L^2}^2\|b^1\|_{L^\infty}+\|b^2\|_{L^\infty}\|\partial_1b\|_{L^2}\|\partial_2b\|_{L^2}\\
	&\leq&\|\partial_1b\|_{L^2}\|u\|_{H^2}\|\nabla u\|_{L^2}+\|\partial_1b\|_{L^2}^2\|b\|_{H^2}+\|\partial_1b\|_{H^1}\|\partial_1b\|_{L^2}\|b\|_{H^1}\\
	&\leq&(\|u\|_{H^2}+\|b\|_{H^2})(\|\nabla u\|_{L^2}^2+\|\partial_1b\|_{H^1}^2).
\end{eqnarray*}
Similarly,
\begin{eqnarray*}
	J_2&=&\langle \partial_1u,b\cdot\nabla u\rangle-\langle \partial_1u,u^1\partial_1b\rangle-\langle \partial_1u,u^2\partial_2b\rangle\\
	&\leq&\|\nabla u\|_{L^2}^2\|b\|_{L^\infty}+\|\partial_1u\|_{L^2}\|u^1\|_{L^\infty}\|\partial_1b\|_{L^2}+\|\partial_1u\|_{L^2}\|u^2\|_{L^\infty}\|\partial_2b\|_{L^2}\\
	&\leq&\|\nabla u\|_{L^2}^2\|b\|_{H^2}+\|\nabla u\|_{L^2}\|u\|_{H^2}\|\partial_1b\|_{L^2}+\|\nabla u\|_{H^1}^2\|\partial_2b\|_{L^2}\\
	&\leq&(\|u\|_{H^2}+\|b\|_{H^2})(\|\nabla u\|_{H^1}^2+\|\partial_1b\|_{L^2}^2).
\end{eqnarray*}
For $J_3$,
\begin{eqnarray*}
	J_3&=&-\int_{\R^2_+}\Delta u\partial_1 {\bf P}bdx=\int_{\R^2_+}\mathcal{A}u\partial_1 bdx\\
	&\leq&\frac12\|\partial_1b\|_{L^2}^2+\frac12\|\mathcal{A} u\|_{L^2}^2.
\end{eqnarray*}
Hence,
\begin{eqnarray}\label{4.10}
	&&\frac{d}{dt}\langle b,\partial_1u\rangle+\frac12\|\partial_1b\|_{L^2}^2\nonumber\\
	&\leq&(\|u\|_{H^2}+\|b\|_{H^2})(\|\nabla u\|_{H^1}^2+\|\partial_1b\|_{H^1}^2)+\|\partial_1u\|_{L^2}^2+\frac12\|\mathcal{A} u\|_{L^2}^2.
\end{eqnarray}

{\bf Step 4.} $\dot{H}^2$ estimate of $b$ and dissipation estimate of $\nabla \partial_1b$.

We apply $\Delta$ to $(\ref{PMHD})_2$ and take the $L^2$ inner product with $-\mathcal{A}b$, then
\begin{eqnarray}\label{4.11}
	\frac12\frac{d}{dt}\|\mathcal{A}b\|_{L^2}^2+\langle \Delta\partial_1u,\mathcal{A}b\rangle=\langle \Delta(u\cdot\nabla b-b\cdot\nabla u),\mathcal{A}b\rangle.
\end{eqnarray}
Thanks to Lemma 2.3, we have
\begin{eqnarray*}
	\langle \Delta\partial_1 u,\mathcal{A}b\rangle=\langle \partial_1\mathcal{A}u,\Delta b\rangle&=&\langle \partial_1{\bf P}(b\cdot\nabla b)-\partial_1{\bf P}(u\cdot\nabla u)-\partial_t\partial_1u,\Delta b\rangle+\langle \partial_1^2b,\Delta b\rangle\\
	&=&\langle \partial_1{\bf P}(b\cdot\nabla b)-\partial_1{\bf P}(u\cdot\nabla u)-\partial_t\partial_1u,\Delta b\rangle+\|\partial_1\nabla b\|_{L^2}^2,
\end{eqnarray*}
where in the last term, we have used that
\begin{eqnarray*}
	\langle \partial_1^2b,\Delta b\rangle = -\int_{\R^2_+} \partial_1 b\partial_1\Delta bdx&=&\int_{\R^2_+} |\partial_1\nabla b|^2dx-\int_{\partial \R^2_+} \partial_1\nabla b\partial_1b\cdot ndS\\
	&=&\int_{\R^2_+} |\partial_1\nabla b|^2dx+\int_{\partial \R^2_+} \partial_1\nabla b\partial_1b^2dS\\
	&=&\int_{\R^2_+} |\partial_1\nabla b|^2dx,
\end{eqnarray*}
 due to $\partial_1b^2|_{\partial \R^2_+}=0$.
Hence, (\ref{4.11}) becomes 
\begin{eqnarray}\label{4.12}
	&&\frac12\frac{d}{dt}\|\mathcal{A}b\|_{L^2}^2+\|\partial_1\nabla b\|_{L^2}^2-\langle \partial_t\nabla u,\partial_1\nabla b\rangle\nonumber\\
	&=&\langle \Delta(u\cdot\nabla b-b\cdot\nabla u),\mathcal{A}b\rangle-
	\langle \partial_1{\bf P}(b\cdot\nabla b)-\partial_1{\bf P}(u\cdot\nabla u),\Delta b\rangle\nonumber\\
	&:=&H_1+H_2.
\end{eqnarray}
Now, let us estimate $H_1$ and $H_2$ one by one. The essential observation is that $\Delta(u\cdot\nabla b-b\cdot\nabla u)\in L^2_{\sigma}$. We have
\begin{eqnarray*}
	H_1&=&\langle {\bf P}\Delta(u\cdot\nabla b), \Delta b\rangle-\langle {\bf P}\Delta(b\cdot\nabla u), \Delta b\rangle\\
	&=&\langle \Delta(u\cdot\nabla b), \Delta b\rangle-\langle \Delta(b\cdot\nabla u), \Delta b\rangle\\
	&:=&H_{11}+H_{12}.
\end{eqnarray*}
Due to $\nabla \cdot u=0$, so $\int_{\R^2_+} u\cdot \nabla \Delta b \Delta bdx=0$ and 
\begin{eqnarray*}
	H_{11}&=&\langle \partial_1^2 b, (\Delta u\cdot\nabla b+\nabla u\cdot\nabla^2b) \rangle+\langle \partial_2^2 b^1, (\Delta u\cdot\nabla b^1+2\partial_i u\cdot\nabla \partial_ib^1) \rangle\\
	&&+\langle \partial_2^2 b^2, (\Delta u\cdot\nabla b^2+2\partial_i u\cdot\nabla \partial_ib^2) \rangle\\
	&:=&H_{111}+H_{112}+H_{113}.
\end{eqnarray*}
\begin{eqnarray*}
	H_{111}&\leq&\|\partial_1^2b\|_{L^2}(\|\Delta u\|_{L^4}\|\nabla b\|_{L^4}+\|\nabla u\|_{L^\infty}\|\nabla^2b\|_{L^2})\\
	&\leq& \|b\|_{H^2}(\|\nabla u\|_{H^2}^2+\|\partial_1b\|_{H^1}^2).
\end{eqnarray*}
Similarly, with the aid of $\partial_2b^2=-\partial_1b^1(\nabla\cdot b=0)$,
\begin{eqnarray*}
	H_{113}=-\langle \partial_2\partial_1 b^1, (\Delta u\cdot\nabla b^2+2\partial_i u\cdot\nabla \partial_ib^2) \rangle 	&\leq& \|b\|_{H^2}(\|\nabla u\|_{H^2}^2+\|\partial_1b\|_{H^1}^2).
\end{eqnarray*}
\begin{eqnarray*}
	H_{112}&=&\langle \partial_2^2b^1, \Delta u\cdot \nabla b^1+2\partial_iu\cdot\nabla \partial_ib^1\rangle\\
	&=&\langle \partial_2^2b^1, \Delta u^1\partial_1b^1+2\partial_iu^1\partial_1\partial_ib^1+2\partial_1u^2\partial_2\partial_1b^1\rangle+\langle \partial_2^2b^1,\partial_1^2u^2\partial_2b^1\rangle\\
	&&+\langle \partial_2^2b^1,\partial_2^2u^2\partial_2b^1+2\partial_2u^2\partial_2^2b^1\rangle\\
	&\leq& \|b\|_{H^2}(\|\nabla u\|_{H^2}^2+\|\partial_1b\|_{H^1}^2)+\int_{\R^2_+} \partial_2b^1\partial_1\partial_2b^1\partial_1\partial_2u^2dx\\
	&&+ \langle \partial_2^2b^1,\partial_2^2u^2\partial_2b^1+2\partial_2u^2\partial_2^2b^1\rangle\\
	&\leq& 2\|b\|_{H^2}(\|\nabla u\|_{H^2}^2+\|\partial_1b\|_{H^1}^2)+ \langle \partial_2^2b^1,\partial_2^2u^2\partial_2b^1+2\partial_2u^2\partial_2^2b^1\rangle,
\end{eqnarray*}
where we have used in the third inequality that 
\begin{eqnarray*}
	\langle \partial_2^2b^1,\partial_1^2u^2\partial_2b^1\rangle
	&=&\frac12\int_{\R^2_+}\partial_1^2u^2\partial_2(\partial_2b^1)^2dx\\
	&=&\frac12\int_{\R^2_+} \partial_2\partial_1 u^2\partial_1(\partial_2 b^1)^2dx+\frac12\int_{\partial\R^2_+}\partial_1u^2\partial_1(\partial_2 b^1)^2dS\\
	&=&\int_{\R^2_+}\partial_2b^1\partial_1\partial_2b^1\partial_1\partial_2u^2dx,
\end{eqnarray*}
where the second term is elimated by $\partial_1u^2|_{\partial \R^2_+}=0$ since $u|_{\partial \R^2_+}=(u^1,u^2)|_{\partial \R^2_+}=0$.\\
For $H_{12}$,
\begin{eqnarray*}
	H_{12}&=&\langle \partial_1^2b,\Delta (b\cdot \nabla u)\rangle+\langle \partial_2^2b^1,\partial_1^2(b\cdot \nabla u^1)\rangle+\langle \partial_2^2b^1, \partial_2^2(b\cdot \nabla u^1)\rangle+\langle \partial_2^2b^2,\Delta (b\cdot \nabla u^2)\rangle\\
	&=&  \langle \partial_1^2b,\Delta (b\cdot \nabla u)\rangle+\langle \partial_2^2b^1,\partial_1^2(b\cdot\nabla u^1)\rangle+\langle \partial_2^2b^1, \partial_2^2(b^1\partial_1 u^1)\rangle
	+\langle \partial_2^2b^1, \partial_2^2(b^2\partial_2 u^1)\rangle\\
	&&-\langle \partial_1  \partial_2b^1,\Delta (b\cdot \nabla u^2)\rangle\\
	&:=&H_{121}+H_{122}+H_{123}+H_{124}+H_{125}.
\end{eqnarray*}
\begin{eqnarray*}
	H_{121}&\leq& \|\nabla \partial_1b\|_{L^2}(\|\nabla^2b\|_{L^2}\|\nabla u\|_{L^\infty}+\|\nabla b\|_{L^4}\|\nabla^2u\|_{L^4}+\|b\|_{L^\infty}\|\nabla^3u\|_{L^2})\\
	&\leq&\|b\|_{H^2}(\|\nabla u\|_{H^2}^2+\|\partial_1b\|_{H^1}^2).
\end{eqnarray*}
Similarly,
\begin{eqnarray*}
	H_{125}\leq \|b\|_{H^2}(\|\nabla u\|_{H^2}^2+\|\partial_1b\|_{H^1}^2).
\end{eqnarray*}
Since $u^1=b^2=0$ on $\partial \R^2_+$,
\begin{eqnarray*}
	H_{122}&=&-\int_{\R^2_+} \partial^2_1\partial_2b^1\partial_2 (b\cdot\nabla u^1)dx+\int_{\partial \R^2_+}\partial_1^2\partial_2b^1 (b\cdot\nabla u^1)dS\\
	&=&-\int_{\R^2_+} \partial_1^2\partial_2b^1\partial_2 (b\cdot\nabla u^1)dx+\int_{\partial \R^2_+}\partial_1^2\partial_2b^1 (b^1\partial_1 u^1+b^2\partial_2u^1)dS\\
	&=&\int_{\R^2_+} \partial_1\partial_2b^1\partial_1\partial_2 (b\cdot\nabla u^1)dx\\
	&\leq&\|b\|_{H^2}(\|\nabla u\|_{H^2}^2+\|\partial_1b\|_{H^1}^2),
\end{eqnarray*}
where we have used the similar estimate to $H_{121}$ in the last inequality.\\
In addition, Lemma 2.1 (2) and $b^1|_{\partial \R^2_+}=0(\Leftarrow b^1(0,\cdot)|_{\partial \R^2_+}=0)$ implies that
\begin{eqnarray*}
	H_{123}&=&\langle \partial_2^2b^1, \partial_2^2(b^1\partial_1 u^1)\rangle\\
	&=&\langle \partial_2^2b^1,b^1\partial_2^2\partial_1u^1\rangle+\langle \partial_2^2b^1,\partial_2^2b^1\partial_1u^1+2\partial_2b^1\partial_2\partial_1u^1\rangle\\
	&\leq&|\int_{\R^2_+} \partial_1\partial_2 b^1\partial_2b^1\partial_2^2u^1+b^1\partial_1\partial_2b^1\partial_2^3u^1+\partial_1b^1\partial_2^2b^1\partial_2^2u^1dx|\\
	&&+\langle \partial_2^2b^1,\partial_2^2b^1\partial_1u^1+2\partial_2b^1\partial_2\partial_1u^1\rangle\\
	&\leq& \|\partial_1b\|_{H^1}\|\partial_2b^1\|_{L^4}\|\partial_2^2u^1\|_{L^4}+\|b^1\|_{L^\infty}\|\partial_1\partial_2b^1\|_{L^2}\|\partial_2^3u^1\|_{L^2}\\
	&&+\|\partial_1b^1\|_{L^4}\|\partial_2^2b^1\|_{L^2}\|\partial_2^2u^1\|_{L^4}
	+\langle \partial_2^2b^1,\partial_2^2b^1\partial_1u^1+2\partial_2b^1\partial_2\partial_1u^1\rangle\\
	&\leq&\|b\|_{H^2}(\|\partial_1b\|_{H^1}^2+\|\nabla u\|_{H^2}^2)+\langle \partial_2^2b^1,\partial_2^2b^1\partial_1u^1+2\partial_2b^1\partial_2\partial_1u^1\rangle.
	%&=&-\langle %\partial_1\partial_2^2b^1,b^1\partial_2^2u^1\rangle-\langle %\partial_2^2b^1,\partial_1b^1\partial_2^2u^1\rangle+\langle %\partial_2^2b^1,\partial_2^2b^1\partial_1u^1+2\partial_2b^1\part%ial_2\partial_1u^1\rangle\\
%	&&
\end{eqnarray*}
\begin{eqnarray*}
H_{124}&=&	\langle \partial_2^2b^1, \partial_2^2(b^2\partial_2 u^1)\rangle=	\langle \partial_2^2b^1, \partial_2^2b^2\partial_2 u^1+2\partial_2b^2\partial_2^2u^1+b^2\partial_2^3u^1\rangle\\
&\leq& \|b\|_{H^2}(\|\partial_2\partial_1b^1\|_{L^2}\|\partial_2u^1\|_{L^\infty}+\|\partial_1b^1\|_{L^4}\|\partial_2^2u^1\|_{L^4}+\|b^2\|_{L^\infty}\|\partial_2^3u^1\|_{L^2})\\
&\leq&\|b\|_{H^2}(\|\nabla u\|_{H^2}^2+\|\partial_1b\|_{H^1}^2).
\end{eqnarray*}
For $H_2$, 
\begin{eqnarray*}
	H_2&=&\langle \partial_1\textbf{P}(u\cdot\nabla u), \Delta b\rangle-\langle \partial_1\textbf{P}(b\cdot\nabla b),\Delta b\rangle\\
	&=&\langle \partial_1\nabla b,\nabla \textbf{P}(u\cdot\nabla u)\rangle-\langle \partial_1\textbf{P}(b\cdot\nabla b),\Delta b\rangle\\
	&:=&H_{21}+H_{22}.
\end{eqnarray*}
Since $u|_{\partial \R^2_+}=0$, it holds\cite{2010H,2012H}
\begin{equation*}
	\textbf{P}(u\cdot\nabla u)=u\cdot\nabla u+\sum_{i,j=1}^2\nabla \mathcal{N}\partial_i\partial_j(u^iu^j).
\end{equation*}
By Lemma 2.4, we have
\begin{eqnarray*}
	H_{21}&\leq& \|\partial_1\nabla b\|_{L^2}\|\nabla \Big( u\cdot\nabla u+\sum \nabla \mathcal{N}\partial_i\partial_j (u^iu^j) \Big)\|_{L^2}\\
	&\leq&\|\partial_1\nabla b\|_{L^2}\|\nabla(u\cdot\nabla u)\|_{L^2}+\|\partial_1\nabla b\|_{L^2}\|\nabla  \sum \nabla \mathcal{N}\partial_i\partial_j(u^iu^j)\|_{L^2}\\
	&\leq&\|\partial_1\nabla b\|_{L^2}(\|\nabla u\|_{L^4}^2+\|u\|_{L^\infty}\|\nabla^2u\|_{L^2}+\|u\|_{L^4}^2+\|\nabla u\|_{L^8}\|\nabla^2u\|_{L^8})\\
	&\leq&\|\partial_1\nabla b\|_{L^2}(\|\nabla u\|_{H^1}^2+\|u\|_{H^2}\|\nabla^2u\|_{L^2}+\|u\|_{L^2}\|\nabla u\|_{L^2}+\|\nabla^3 u\|_{L^2}^{\frac34}\|\nabla^2u\|_{L^2}\|\nabla u\|_{L^2}^{\frac14})\\
	&\leq&\| u\|_{H^2}(\|\partial_1 b\|_{H^1}^2+\|\nabla u\|_{H^2}^2).
\end{eqnarray*}
 For $H_{22}$,
\begin{eqnarray*}
	H_{22}&\leq&|\int_{\R^2_+}\partial_1( b\cdot\nabla b)\mathcal{A}bdx|
	=|\int_{\R^2_+}\partial_1 \textbf{P}(b\cdot\nabla b)(\partial_1^2+\partial_2^2) b dx|\\
	&\leq&|\int_{\R^2_+}\partial_1 \textbf{P}(b\cdot\nabla b)\partial_1^2 b dx|+|\int_{\R^2_+}\partial_1 \textbf{P}(b\cdot\nabla b)\partial_2^2 b dx|\\
	&:=&H_{221}+H_{222}.
\end{eqnarray*}
\begin{eqnarray*}
	H_{221}&\leq&\|\textbf{P}\partial_1(b\cdot\nabla b)\|_{L^2}\|\partial_1^2b\|_{L^2}\leq \|\partial_1(b\cdot\nabla b)\|_{L^2}\|\partial_1^2b\|_{L^2}\\
	&\leq&\|\partial_1b\|_{L^4}\|\nabla b\|_{L^4}\|\partial_1^2b\|_{L^2}+\|b\|_{L^\infty}\|\partial_1\nabla b\|_{L^2}\|\partial_1^2b\|_{L^2}\\
	&\leq&\|b\|_{H^2}\|\partial_1b\|_{H^1}^2.
\end{eqnarray*}
%It follows from $b^2=0$ and $b^1=0$ on $\partial \R^2_+$ that
%\begin{equation*}
%	\textbf{P}(b\cdot\nabla b)=b\cdot\nabla b+\sum_{i,j=1}^2\nabla \mathcal{N}\partial_i\partial_j(b^ib^j).
%\end{equation*}
%Similar to $H_{21}$,
%\begin{eqnarray*}
%{\color{red}H_{222}}	&=&|\int_{\R^2_+}\partial_2 \Big(b\cdot\nabla b+\sum \nabla \mathcal{N}\partial_i\partial_j(b^ib^j)\Big)\partial_1\partial_2 bdx|\\
%&\leq&|\int_{\R^2_+}\partial_2 (b\cdot\nabla b)\partial_1\partial_2 bdx|+|\int_{\R^2_+}\partial_2\sum \nabla \mathcal{N}\partial_i\partial_j(b^ib^j)\partial_1\partial_2 bdx|\\
%&=&H_{2221}+{\color{red}H_{2222}}\\
%	&\leq&\|\mathcal{A}b\|_{L^2}(\|\partial_1b\|_{L^4}\|\nabla b\|_{L^4}+\|b\|_{L^\infty}\|\partial_1b\|_{H^1})\\
%	&\leq&\|b\|_{H^2}\|\partial_1b\|_{H^1}\|b\|_{H^2}^{\frac34}\|\partial_1b\|_{H^1}^{\frac14}+\|b\|_{H^2}^2\|\partial_1b\|_{H^1}\\
%	&\leq&\|b\|_{H^2}\|\partial_1b\|_{H^1}(\|b\|_{H^2}+\|\partial_1b\|_{H^1})+\|b\|_{H^2}^2\|\partial_1b\|_{H^1}\\
%	&\leq&(\|b\|_{H^2}^2+\|b\|_{H^2})(\|\partial_1b\|_{H^1}+\|\partial_1b\|_{H^1}^2).
%\end{eqnarray*}
By lemma 2.1, 
\begin{eqnarray*}
H_{222}&\leq& |\int_{\R^2_+}\partial_1\textbf{P}( b\cdot\nabla b)\partial_2^2bdx|\\
	&\leq&\|\partial_2^2b\|_{L^2}(\|\partial_1b\|_{L^4}\|\nabla b\|_{L^4}+\|b\|_{L^\infty}\|\partial_1\nabla b\|_{L^2})\\
	&\leq&\|\partial_2^2b\|_{L^2}\|\partial_1b\|_{H^1}\|b\|_{H^2}^{\frac34}\|\partial_1b\|_{H^1}^{\frac14}+\|\partial_2^2b\|_{L^2}\|b\|_{H^1}^{\frac12}\|\partial_1b\|_{H^1}^{\frac12}\|\partial_1b\|_{H^1}\\
	&\leq&\|b\|_{H^2}^2\|\partial_1b\|_{H^1}^2+\tilde{\epsilon}(\|b\|_{H^2}^2+\|\partial_1b\|_{H^1}^2)+\tilde{\epsilon}(\|b\|_{H^1}^2+\|\partial_1b\|_{H^1}^2)\\
	&\leq&\|b\|_{H^2}^2\|\partial_1b\|_{H^1}^2+\tilde{\epsilon}(\|b\|_{H^2}^2+\|\partial_1b\|_{H^1}^2),
\end{eqnarray*}
where $0<\tilde{\epsilon}\ll 1$ is a constant to be determined.
Hence, (\ref{4.12}) becomes
\begin{eqnarray}\label{4.13}
	&&\frac{d}{dt}\|\mathcal{A}b\|_{L^2}^2+\|\partial_1\nabla b\|_{L^2}^2-\| \partial_t\nabla u\|_{L^2}^2\nonumber\\
	&\leq&C(\|u\|_{H^1}^2+\|b\|_{H^2}+\|b\|_{H^2}^2)(\|\nabla u\|_{H^2}^2+\|\nabla u\|_{H^2}+\|\partial_1b\|_{H^1}+\|\partial_1b\|_{H^1}^2)\nonumber\\
	&&+ \langle \partial_2^2b^1,\partial_2\partial_1u^1\partial_2b^1\rangle + \langle \partial_2^2b^1,\partial_2u^2\partial_2^2b^1\rangle  \nonumber\\
	&:=&C(\|u\|_{H^1}^2+\|u\|_{H^2}+\|b\|_{H^2}+\|b\|_{H^2}^2)(\|\nabla u\|_{H^2}^2%+\|\nabla u\|_{H^2}
	+\|\partial_1b\|_{H^1}^2)+\tilde{\epsilon}(\|b\|_{H^2}^2+\|\partial_1b\|_{H^1}^2)\nonumber\\
	&&+\mathcal{T}_1+\mathcal{T}_2. 
\end{eqnarray}
{\bf Step 4.1.} The estimate of $\mathcal{T}_1$.\\
By $\partial_2u^2=-\partial_1u^1$ and equation $(\ref{PMHD})_2$, we have
\begin{eqnarray}\label{4.14}
	\mathcal{T}_1&=&\langle \partial_2^2b^1\partial_2b^1,\partial_2(-\partial_1u^1)\rangle=
	\langle \partial_2^2b^1\partial_2b^1,\partial_2(-b^1_t+b\cdot\nabla u^1-u\cdot\nabla b^1)\rangle\nonumber\\
	&=&-\frac12\frac{d}{dt}\langle \partial_2^2b^1,(\partial_2b^1)^2\rangle+\frac12\langle \partial_2^2b^1_t,(\partial_2b^1)^2\rangle+\langle \partial_2^2b^1\partial_2b^1,\partial_2(b\cdot\nabla u^1)\rangle\nonumber\\
	&&-\langle \partial_2^2b^1\partial_2b^1,\partial_2u\cdot\nabla b^1+u\cdot\nabla\partial_2b^1\rangle.
\end{eqnarray}
Then by using equation $(\ref{PMHD})_2$ again, we obtain
\begin{eqnarray*}
	\langle (\partial_2b^1)^2,\partial_2^2b_t^1\rangle&=&\langle (\partial_2b^1)^2,\partial_2^2(\partial_1u^1-u\cdot b^1+b\cdot u^1)\rangle\\
	&=&\langle (\partial_2b^1)^2,\partial_2^2\partial_1u^1\rangle-\langle (\partial_2b^1)^2,\partial_2^2u\cdot\nabla b^1+2\partial_2u\cdot \nabla \partial_2b^1+u\cdot\nabla \partial_2^2b^1\rangle\\
	&&+\langle (\partial_2b^1)^2,\partial_2^2(b\cdot\nabla u^1) \rangle\\
	&=&\langle (\partial_2b^1)^2,\partial_2^2\partial_1u^1\rangle-\langle (\partial_2b^1)^2,\partial_2^2u\cdot\nabla b^1+2\partial_2u\cdot \nabla \partial_2b^1\rangle\\
	&& +\langle (\partial_2b^1)^2,u\cdot\nabla(\partial_2b^1)^2,\partial_2^2b^1\rangle+\langle (\partial_2b^1)^2,\partial_2^2(b\cdot\nabla u^1)\rangle,
\end{eqnarray*}
which together with (\ref{4.14}) yields that
\begin{eqnarray}\label{4.15}
	&&\frac12\frac{d}{dt}\langle \partial_2^2b^1,(\partial_2b^1)^2\rangle+\mathcal{T}_1\nonumber\\
	&=&\frac12 \langle (\partial_2b^1)^2,\partial_2^2\partial_1u^1\rangle-\frac12 \langle (\partial_2b^1)^2,\partial_2^2u\cdot\nabla b^1\rangle+\langle (\partial_2b^1)^2, \partial_2u\cdot\nabla \partial_2b^1\rangle\nonumber\\
	&&+\langle \partial_2^2b^1\partial_2b^1,\partial_2u\cdot\nabla b^1\rangle-\langle \partial_2^2b^1\partial_2b^1,\partial_2(b\cdot\nabla u^1)\rangle+\frac12\langle (\partial_2b^1)^2,\partial_2^2(b\cdot\nabla u^1)\rangle\nonumber\\
	&:=&L_i(i=1,2,\cdots,6).
\end{eqnarray}
We now estimate the right terms of (\ref{4.15}) one by one. First,
\begin{eqnarray*}
L_1&=&	\frac12 \langle (\partial_2b^1)^2,\partial_2^2\partial_1u^1\rangle=	- \langle \partial_2b^1\partial_2\partial_1b^1,\partial_2^2u^1\rangle\leq \|\partial_2b^1\|_{L^4}\|\partial_2\partial_1b^1\|_{L^2}\|\partial_2^2u^1\|_{L^4}\\
	&\leq&\|b\|_{H^2}(\|\nabla u\|_{H^2}^2+\|\partial_1b\|_{H^1}^2).
\end{eqnarray*}
Similarly,
\begin{eqnarray*}
	L_2\leq \|b\|_{H^2}(\|\nabla u\|_{H^2}^2+\|\partial_1b\|_{H^1}^2).
\end{eqnarray*}
\begin{eqnarray*}
	L_3&=&\langle (\partial_2b^1)^2,\partial_2u^1\partial_1\partial_2b^1\rangle+\langle (\partial_2b^1)^2,\partial_2u^2\partial_2\partial_2b^1\rangle\\
	&\leq&\|\partial_2u^1\|_{L^\infty}\|\partial_1\partial_2b^1\|_{L^2}\|(\partial_2b^1)^2\|_{L^2}+\frac13\langle \partial_2u^2, \partial_2(\partial_2b^1)^3\rangle\\
	&\leq&\|\nabla u\|_{H^2}\|\partial_1b\|_{H^1}\|b\|_{H^2}^2-\frac13\langle \partial_2u^1, \partial_1(\partial_2b^1)^3\rangle\\
	&\leq& \|b\|_{H^2}^2(\|\nabla u\|_{H^2}^2+\|\partial_1b\|_{H^1}^2)+\|\partial_2u^1\|_{L^\infty}\|\partial_1b\|_{H^1}\|(\partial_2b^1)^2\|_{L^2}\\
	&\leq&  C\|b\|_{H^2}^2(\|\nabla u\|_{H^2}^2+\|\partial_1b\|_{H^1}^2),
\end{eqnarray*}
where we have used that $u=0$ $(\partial_1u^1=0)$ on $\partial\R^2_+$ in second inequality.
\begin{eqnarray*}
	L_4&=&\langle \partial_2^2b^1\partial_2b^1,\partial_2u^1\partial_1b^1+\partial_2u^2\partial_2b^1\rangle\\
	&=&\frac12\langle \partial_2(\partial_2b^1)^2,\partial_2u^1\partial_1b^1\rangle-\frac13\langle \partial_2(\partial_2b^1)^3,\partial_1u^1\rangle\\
	&\leq&\frac12\langle (\partial_2b^1)^2,\partial_2^2u^1\partial_1b^1+\partial_2u^1\partial_1\partial_2b^1\rangle+\langle (\partial_2b^1)^2\partial_1\partial_2b^1,\partial_2u^1\rangle\\
	&\leq&\|(\partial_2b^1)^2\|_{L^2}\|\partial_2^2u^1\|_{L^4}\|\partial_1b^1\|_{L^4}+\|(\partial_2b^1)^2\|_{L^2}\|\partial_2u^1\|_{L^\infty}\|\partial_1\partial_2b^1\|_{L^2}\\
	&\leq& \|b\|_{H^2}^2(\|\nabla u\|_{H^2}^2+\|\partial_1b\|_{H^1}^2).
\end{eqnarray*}
We rewrite $L_5$ and $L_6$ and use $u^1=0(\partial_1u^1=0)$ on $\partial\R^3_+$ to have
\begin{eqnarray*}
	L_5+L_6&=&\langle \partial_2^2b^1\partial_2b^1,\partial_2b^1\partial_1u^1+\partial_2b^2\partial_2u^1+b\cdot\nabla \partial_2u^1\rangle\\
	&&-\frac12\langle (\partial_2b^1)^2,\partial_2^2b^1\partial_1u^1+2\partial_2b^1\partial_2\partial_1u^1+\partial_2^2b^2\partial_2u^1+2\partial_2b^2\partial_2^2u^1+b\cdot\nabla \partial_2^2u^1\rangle\\
	&=&\frac13\langle \partial_2(\partial_2b^1)^3,\partial_1u^1\rangle+\{ \langle \partial_2^2b^1\partial_2b^1,\partial_1b^1\partial_2u^1\rangle+\langle \partial_2^2b^1\partial_2b^1,b\cdot\nabla \partial_2u^1\rangle\\
	&&-\frac12\langle (\partial_2b^1)^2,2\partial_2b^1\partial_2\partial_1u^1-\partial_2\partial_1b^1\partial_2u^1-2\partial_1b^1\partial_2^2u^1+b\cdot\nabla \partial_2^2u^1\rangle \}\\
	&:=&\langle \partial_2\partial_1b^1(\partial_2b^1)^2,\partial_2u^1\rangle+L_{61}\\
	&\leq& \|\partial_1\partial_2b^1\|_{L^2}\|(\partial_2b^1)^2\|_{L^2}\|\partial_2u^1\|_{L^\infty}+L_{61}\\
	&\leq& \|\partial_1b\|_{H^1}\|\partial_2b^1\|_{L^4}^2\|\nabla u\|_{H^2}+L_{61}\\
	&\leq&\|b\|_{H^2}^2(\|\partial_1b\|_{H^1}^2+\|\nabla u\|_{H^2}^2)+L_{61},
\end{eqnarray*}
\begin{eqnarray*}
	L_{61}\leq \|b\|_{H^2}^2(\|\nabla u\|_{H^2}^2+\|\partial_1b\|_{H^1}^2).
\end{eqnarray*}
Hence, (\ref{4.15}) becomes 
\begin{eqnarray}\label{4.16}
	&&\frac{d}{dt}\langle \partial_2^2b^1,(\partial_2b^1)^2\rangle+\mathcal{T}_1\nonumber\\
&\leq&\|b\|_{H^2}^2(\|\nabla u\|_{H^2}^2+\|\partial_1b\|_{H^1}^2).
\end{eqnarray}
{\bf Step 4.2. The estimate of $\mathcal{T}_2$}.\\
We use $\partial_2u^2=-\partial_1u^1$ and $(\ref{PMHD})_2$ to obtain
\begin{eqnarray*}
	\mathcal{T}_2&=&-\langle \partial_2^2b^1,\partial_1u^1\partial_2^2b^1\rangle\\
	&=&-\langle (\partial_2^2b^1)^2,\partial_tb^1\rangle+\langle (\partial_2^2b^1)^2,b\cdot\nabla u^1-u\cdot\nabla b^1\rangle\\
	&=&-\frac{d}{dt}\langle b^1, (\partial_2^2b^1)^2\rangle+2\langle b^1,\partial_2^2b^1\partial_2^2b^1_t\rangle+\langle (\partial_2^2b^1)^2,b\cdot\nabla u^1\rangle+\langle b^1,u\cdot\nabla (\partial_2^2b^1)^2\rangle.
\end{eqnarray*}
In order to estimate the second term on the right hand side of above equation, we employ $(\ref{PMHD})_2$ again to have
\begin{eqnarray*}
	\langle  b^1,\partial_2^2b^1\partial_2^2b^1_t\rangle&=&\langle b^1,\partial_2^2b^1\partial_2^2(\partial_1u^1-u\cdot\nabla b^1+b\cdot\nabla u^1)\rangle\\
	&=&\langle b^1,\partial_2^2b^1\partial_2^2\partial_1u^1\rangle-\langle b^1\partial_2^2b^1,\partial_2^2u\cdot\nabla b^1+2\partial_2u\cdot\nabla\partial_2b^1+u\cdot\nabla\partial_2^2b^1\rangle\\
	&&+\langle b^1,\partial_2^2b^1\partial_2^2(b\cdot\nabla u^1)\rangle\\
	&=&\langle b^1,\partial_2^2b^1\partial_2^2\partial_1u^1\rangle-\langle b^1\partial_2^2b^1,\partial_2^2u\cdot\nabla b^1+2\partial_2u\cdot\nabla\partial_2b^1\rangle\\
	&&-\frac12\langle b^1,u\cdot\nabla(\partial_2^2b^1)^2\rangle+\langle b^1,\partial_2^2b^1\partial_2^2(b\cdot\nabla u^1)\rangle.
\end{eqnarray*}
Together with above equation we obtain
\begin{eqnarray}\label{4.17}
	&&\frac{d}{dt}\langle b^1,(\partial_2^2b^1)^2\rangle+\mathcal{T}_2\nonumber\\
	&=&2\langle b^1,\partial_2^2b^1\partial_2^2\partial_1u^1\rangle-2\langle b^1\partial_2^2b^1,\partial_2^2u\cdot\nabla b^1\rangle-4\langle b^1\partial_2^2b^1,\partial_2u\cdot\nabla\partial_2b^1\rangle\nonumber\\
	&&+2\langle b^1,\partial_2^2b^1\partial_2^2(b\cdot\nabla u^1)\rangle-\langle b\cdot\nabla u^1,(\partial_2^2b^1)^2\rangle\nonumber\\
	&=&K_i(i=1,2,3,\cdots,5).
\end{eqnarray}
We now estimate $K_i$ one by one.
By H\"{o}lder's inequality and $b^1=0$ on $\partial \R^2_+$, we have
\begin{eqnarray*}
	K_1&\leq&|\int_{\R^2_+}\partial_1b^1\partial_2^2b^1\partial_2^2u^1+\partial_1\partial_2b^1\partial_2b^1\partial_2^2u^1+b^1\partial_1\partial_2b^1\partial_2^3u^1dx| \\
	&\leq&\|\partial_1b^1\|_{L^4}\|\partial_2^2b^1\|_{L^2}\|\partial_2^2u^1\|_{L^4}+\|\partial_1\partial_2b^1\|_{L^2}\|\partial_2b^1\|_{L^4}\|\partial_2^2u^1\|_{L^4}\\
	&&+\|b^1\|_{L^\infty}\|\partial_1\partial_2b^1\|_{L^2}\|\partial_2^3u^1\|_{L^2}\\
	&\leq&\|b\|_{H^2}(\|\partial_1b\|_{H^1}^2+\|\nabla u\|_{H^2}^2).
\end{eqnarray*}
\begin{eqnarray*}
	K_2&\leq&\|b^1\|_{L^\infty}\|\partial_2^2b^1\|_{L^2}\|\partial_2^2u\|_{L^2_{x^1}L^\infty_{x^2}}\|\nabla b\|_{L^2_{x^2}L^\infty_{x^1}}\\
	&\leq&\|b^1\|_{H^1}^{\frac12}\|\partial_1b^1\|_{H^1}^{\frac12}\|\partial_2^2b^1\|_{L^2}\|\partial_2^2u\|_{H^1}\|\nabla b\|_{L^2}^{\frac12}\|\nabla\partial_1b\|_{L^2}^{\frac12}\\
	&\leq&\|b\|_{H^2}^2(\|\nabla u\|_{H^2}^2+\|\partial_1b\|_{H^1}^2).
\end{eqnarray*}
\begin{eqnarray*}
	K_3&=&-4\langle b^1\partial_2^2b^1,\partial_2u^1\partial_1\partial_2b^1\rangle+4\langle b^1(\partial_2^2b^1)^2,\partial_2u^2\rangle :=K_{31}+4\mathcal{T}_{21}.
\end{eqnarray*}
It is easy to find that 
\begin{equation*}
	K_{31}\leq \|b\|_{L^\infty}\|\nabla^2b\|_{L^2}\|\nabla u\|_{L^\infty}\|\partial_1\nabla b\|_{L^2}\leq \|b\|_{H^2}^2(\|\nabla u\|_{H^2}^2+\|\partial_1b\|_{H^1}^2).
\end{equation*}
The trouble term $\mathcal{T}_{21}$ is estimated by $\partial_2u^2=-\partial_1u^1$ that
\begin{eqnarray*}
	\mathcal{T}_{21}&=&-\langle b^1_t,b^1(\partial_2^2b^1)^2\rangle-\langle u\cdot\nabla b^1-b\cdot\nabla u^1,b^1(\partial_2^2b^1)^2\rangle\\
	&=&-\frac{1}{2}\frac{d}{dt}\langle (b^1)^2,(\partial_2^2b^1)^2\rangle+\langle (b^1)^2,\partial_2^2b^1\partial_2^2b^1_t\rangle-\langle u\cdot\nabla b^1-b\cdot\nabla u^1,b^1(\partial_2^2b^1)^2 \rangle\\
	&=&-\frac{1}{2}\frac{d}{dt}\langle (b^1)^2,(\partial_2^2b^1)^2\rangle+\langle (b^1)^2,\partial_2^2b^1\partial_2^2(\partial_1u^1-u\cdot\nabla b^1+b\cdot\nabla u^1)\rangle\\
	&&-\langle u\cdot\nabla b^1-b\cdot\nabla u^1,b^1(\partial_2^2b^1)^2 \rangle.
\end{eqnarray*}
Then it deduces from above equations that
\begin{eqnarray*}
&&	\frac12\frac{d}{dt}\langle (b^1)^2,(\partial_2^2b^1)^2\rangle+\mathcal{T}_{21}\\
	&=&\langle (b^1)^2,\partial_2^2b^1\partial_2^2(\partial_1u^1-u\cdot\nabla b^1+b\cdot\nabla u^1)\rangle
	-\langle u\cdot\nabla b^1-b\cdot\nabla u^1,b^1(\partial_2^2b^1)^2 \rangle\\
	&=&\langle (b^1)^2\partial_2^2b^1,\partial_2^2\partial_1u^1-\partial_2^2u\cdot\nabla b^1-2\partial_2u\cdot\nabla \partial_2b^1+\partial_2^2(b\cdot\nabla u^1)\rangle
	+\langle b\cdot\nabla u^1,b^1(\partial_2^2b^1)^2 \rangle,
\end{eqnarray*}
where we have used that by $\nabla\cdot u=0$
$$ \langle (b^1)^2\partial_2^2b^1,u\cdot\nabla \partial_2^2b^1\rangle+\langle u\cdot\nabla b^1,b^1(\partial_2^2b^1)^2\rangle=0.$$
Thus by lemma 2.1, one has
\begin{eqnarray}\label{4.18z}
&&	\frac12\frac{d}{dt}\langle (b^1)^2,(\partial_2^2b^1)^2\rangle+\mathcal{T}_{21} \nonumber\\
	&\leq&\|b\|_{L^\infty}^2\|\nabla^2b\|_{L^2}(\|\nabla^3u\|_{L^2}+\|\nabla^2u\|_{L^4}\|\nabla b\|_{L^4}+\|\nabla u\|_{L^\infty}\|\nabla^2b\|_{L^2}+\|b\|_{L^\infty}\|\nabla^3u\|_{L^2})\nonumber\\
	&\leq&\|b\|_{H^2}^2\|\partial_1b\|_{H^1}(\|\nabla u\|_{H^2}+\|b\|_{H^2}\|\nabla u\|_{H^2})\nonumber\\
	&\leq&(\|b\|_{H^2}^2+\|b\|_{H^2}^3)(\|\nabla u\|_{H^2}^2+\|\partial_1b\|_{H^1}^2).
\end{eqnarray}
Finally, for $K_4$ and $K_5$, 
\begin{eqnarray*}
	K_4+K_5&=&\langle b^1,\partial_2^2b^1\partial_2^2(b^1\partial_1u^1)\rangle+\langle b^1,\partial_2^2b^1\partial_2^2(b^2\partial_2u^1)\rangle+
\langle b^1\partial_1u^1,(\partial_2^2b^1)^2\rangle\\
&&+\langle b^2\partial_2u^1,(\partial_2^2b^1)^2\rangle	\\
&=&3\langle b^1(\partial_2^2b^1)^2,\partial_1u^1\rangle-\Big( 2\langle b^1\partial_2^2b^1,2\partial_2b^1\partial_2\partial_1u^1-\partial_2\partial_1b^1\partial_2u^1-2\partial_1b^1\partial_2^2u^1 \rangle\\
&&+2\langle b^1\partial_2^2b^1,b\cdot\nabla \partial_2^2u^1\rangle+\langle b^2\partial_2u^1,(\partial_2^2b^1)^2 \rangle\Big)\\
&:=&3\mathcal{T}_{21}+K_{41}.
\end{eqnarray*}
It follows from lemma 2.1 and Sobolev's embedding that
\begin{eqnarray*}
	K_{41}&\leq& \|b^1\|_{L^\infty}\|\partial_2^2b^1\|_{L^2}\Big( \|\partial_2 b^1\|_{L^2_{x^2}L^\infty_{x^1}}\|\partial_2\partial_1u^1\|_{L^2_{x^2}L^\infty_{x^1}}+\|\partial_2\partial_1b^1\|_{L^2}\|\partial_2u^1\|_{L^\infty}\\
	&&+\|\partial_1b^1\|_{L^4}\|\partial_2^2u^1\|_{L^4}\Big)+\|b\|_{L^\infty}^2\|\partial_2^2b^1\|_{L^2}\|\nabla \partial_2^2u^1\|_{L^2}+\|b^2\|_{L^\infty}\|\partial_2u^1\|_{L^\infty}\|\partial_2^2b^1\|_{L^2}^2\\
	&\leq& \|b\|_{H^2}^2(\|\nabla u\|_{H^2}^2+\|\partial_1 b\|_{H^1}^2).
\end{eqnarray*}
Hence, (\ref{4.17}) becomes 
\begin{eqnarray}\label{4.18}
	&&\frac{d}{dt}\langle b^1,(\partial_2^2b^1)^2\rangle+\mathcal{T}_2\nonumber\\
	&\leq&7\mathcal{T}_{21}+C(\|b\|_{H^2}^2+\|b\|_{H^2})(\|\partial_1 b\|_{H^1}^2+\|\nabla u\|_{H^2}^2).
\end{eqnarray}	
Combining (\ref{4.18}) with (\ref{4.18z}) yields 
\begin{eqnarray*}
	\frac{7}{2}\frac{d}{dt}\langle (b^1)^2,(\partial_2^2b^1)^2 \rangle+\frac{d}{dt}\langle b^1,(\partial_2^2b^1)^2 \rangle +\mathcal{T}_2\leq C(\|b\|_{H^2}+\|b\|_{H^2}^3)(\|\partial_1b\|_{H^1}^2+\|\nabla u\|_{H^2}^2).
\end{eqnarray*}

{\bf Step 4.3.} 
Combining (\ref{4.13})(\ref{4.16}) with (\ref{4.18}) and considering the interpolation inequality, we have
\begin{eqnarray}\label{4.19}
	&&\frac{d}{dt}\|\mathcal{A}b\|_{L^2}^2+\frac{d}{dt}(\langle \partial_2^2b^1,(\partial_2b^1)^2\rangle+\langle b^1,(\partial_2^2b^1)^2\rangle)+\|\partial_1\nabla b\|_{L^2}^2-\|\partial_t\nabla u\|_{L^2}^2\nonumber\\
	&\leq& C(\|u\|_{H^1}^2+\|(u,b)\|_{H^2}
	%+\|b\|_{H^2}^2
	+\|b\|_{H^2}^{\frac52})(\|\nabla u\|_{H^2}^2%+\|\nabla u\|_{H^2}
	+\|\partial_1b\|_{H^1}^2)+\tilde{\epsilon}(\|b\|_{H^2}^2+\|\partial_1b\|_{H^1}^2).
\end{eqnarray}
{\bf Step 5. Dissipation estimate of $(u_t,b_t)$}\\
Taking the $L^2$ product of equation $(\ref{PMHD})_1$ and $(\ref{PMHD})_2$ with $u_t$ and $b_t$, respectively, we obtain
\begin{eqnarray}\label{4.20}
	&&\frac12\frac{d}{dt}\|\nabla u\|_{L^2}^2+\|(u_t,b_t)\|_{L^2}^2\nonumber\\
	&=&-\langle u_t,u\cdot\nabla u\rangle+\langle u_t,b\cdot\nabla b\rangle-\langle b_t,u\cdot\nabla b\rangle+\langle b_t,b\cdot \nabla u\rangle+(\langle u_t,\partial_1b\rangle+\langle b_t,\partial_1u\rangle)\nonumber\\
	&=&M_1+M_2+M_3+M_4+M_5,
\end{eqnarray}
where the pressure disappears since 
$$\int_{\R^2_+}\nabla \pi u_tdx=-\int_{\R^2_+}\pi \nabla \cdot u_tdx+\int_{\partial \R^2_+}\pi u_t\cdot ndS=0,$$
due to $u(t,\cdot)=0$ on $\partial \R^2_+$ and $\nabla \cdot u(t,\cdot)=0$ in $\R^2$.\\
We now estimate $M_i(i=1,2,3,\cdots,5)$ one by one.\\
\begin{eqnarray*}
	M_1&\leq& \|u_t\|_{L^2}\|u\|_{L^\infty}\|\nabla u\|_{L^2}\leq \|u_t\|_{L^2}\|u\|_{H^2}\|\nabla u\|_{L^2}\\
	&\leq&\|u\|_{H^2}(\|\nabla u\|_{L^2}^2+\|u_t\|_{L^2}^2);
\end{eqnarray*}
\begin{eqnarray*}
	M_2&=& \langle u_t, b^1\partial_1b\rangle+\langle u_t,b^2\partial_2b\rangle\\
	&\leq& \|u_t\|_{L^2}\|b^1\|_{L^\infty}\|\partial_1b\|_{L^2}+\|u_t\|_{L^2}\|b^2\|_{L^\infty}\|\partial_2b\|_{L^2}\\
	&\leq&\|u_t\|_{L^2}\|b\|_{H^2}\|\partial_1b\|_{L^2}+\|u_t\|_{L^2}\|\partial_1b\|_{H^1}\|b\|_{H^1}\\
	&\leq&\|b\|_{H^2}(\|u_t\|_{L^2}^2+\|\partial_1b\|_{H^1}^2);
\end{eqnarray*}
\begin{eqnarray*}
	M_3&=&-\langle b_t,u^1\partial_1b\rangle-\langle b_t,u^2\partial_2b\rangle\\
	&\leq& \|b_t\|_{L^2}\|u^1\|_{L^\infty}\|\partial_1b\|_{L^2}+\|b_t\|_{L^2}\|u^2\|_{L^\infty}\|\partial_2b\|_{L^2}\\
	&\leq&\|b_t\|_{L^2}\|u\|_{H^2}\|\partial_1b\|_{L^2}+\|b_t\|_{L^2}\|\nabla u\|_{H^1}\|b\|_{H^1}\\
	&\leq&(\|u\|_{H^2}+\|b\|_{H^1})(\|b_t\|_{L^2}^2+\|\partial_1b\|_{L^2}^2+\|\nabla u\|_{H^1}^2);
\end{eqnarray*}
\begin{eqnarray*}
	M_4&\leq&\|b_t\|_{L^2}\|b\|_{L^\infty}\|\nabla u\|_{L^2}\leq \|b_t\|_{L^2}\|b\|_{H^2}\|\nabla u\|_{L^2}\\
	&\leq&\|b\|_{H^2}(\|b_t\|_{L^2}^2+\|\nabla u\|_{L^2}^2);
\end{eqnarray*}
\begin{eqnarray*}
	M_5\leq \frac12\|u_t\|_{L^2}^2+\frac12\|b_t\|_{L^2}^2+\|\partial_1b\|_{L^2}^2+\|\nabla u\|_{L^2}^2.
\end{eqnarray*}
Hence, combining these estimates with (\ref{4.20}), we have
\begin{eqnarray}\label{4.21}
	&&\frac{d}{dt}\|\nabla u\|_{L^2}^2+\|(u_t,b_t)\|_{L^2}^2-\|\partial_1b\|_{L^2}^2-\|\nabla u\|_{L^2}^2\nonumber\\
	&\leq& (\|u\|_{H^2}+\|b\|_{H^2})(\|b_t\|_{L^2}^2+\|u_t\|_{L^2}^2+\|\partial_1b\|_{H^1}^2+\|\nabla u\|_{H^1}^2).
\end{eqnarray}

{\bf Step 6. Dissipation estimate of $\nabla u_t$}\\
Applying $\partial_t$ to equation $(\ref{PMHD})_1$ and taking the $L^2$ inner product  of resulting equation with $u_t$, we have
\begin{eqnarray}\label{4.22}
	\frac12\frac{d}{dt}\|u_t\|_{L^2}^2+\|\nabla u_t\|_{L^2}^2&=&-\langle u_t, \partial_t(u\cdot\nabla u)\rangle+\langle u_t,\partial_t(b\cdot\nabla b)\rangle+\langle \partial_1b_t,u_t\rangle\nonumber\\
	&:=&Q_1+Q_2+Q_3.
\end{eqnarray}
\begin{eqnarray*}
	Q_1&=&\langle u,u_t\cdot\nabla u_t,u\rangle-\langle u_t,u\cdot\nabla u_t\rangle
	\leq\|\nabla u_t\|_{L^2}\|u_t\|_{L^2}\|u\|_{L^\infty}\\
	&\leq&\|u\|_{H^2}(\|u_t\|_{L^2}^2+\|\nabla u_t\|_{L^2}^2);
\end{eqnarray*}
\begin{eqnarray*}
	Q_2&=&\langle u_t, b_t\cdot\nabla b\rangle+\langle u_t,b\cdot\nabla b_t\rangle=-\langle b_t\cdot \nabla u_t,b\rangle-\langle b\cdot\nabla u_t,b_t\rangle\\
	&\leq&\|b_t\|_{L^2}\|\nabla u_t\|_{L^2}\|b\|_{L^\infty}
	\leq \|b\|_{H^2}(\|b_t\|_{L^2}^2+\|\nabla u_t\|_{L^2}^2);
\end{eqnarray*}
\begin{eqnarray*}
	Q_3&=&-\int \partial_1u_tb_tdx\leq \frac12\|\nabla u_t\|_{L^2}^2+\|b_t\|_{L^2}^2.
\end{eqnarray*}
Combining these estimates with (\ref{4.22}) yields
\begin{eqnarray}
	\frac{d}{dt}\|u_t\|_{L^2}^2+\|\nabla u_t\|_{L^2}^2-\|b_t\|_{L^2}^2
	\leq (\|u\|_{H^2}+\|b\|_{H^2})(\|\nabla u_t\|_{L^2}^2+\|b_t\|_{L^2}^2). 
\end{eqnarray}

{\bf Step 7. $\dot{H}^2$ estimate of $u$ and the dissipation estimate of $\nabla^2u$.}\\
We rewrite $(\ref{PMHD})_1$ as
\begin{equation}\label{4.24}
	\left\{
	\begin{array}{l}
	-\Delta u +\nabla \pi =	-u_t - u\cdot \nabla u +\partial_1b  + b\cdot\nabla b, ~~x\in \R^2_+,t>0 \\
		\nabla\cdot u  =0, ~~x\in \R^2_+,t>0 \\
		u=0, ~~x\in \partial\R^2_+,t>0\\
		u(x,0)=u_0(x),~~x\in \R^2_+.
	\end{array} \right.
\end{equation}
It follows from the estimates for the steady Stokes flow in half space (Lemma IV.3.2, \cite{2011G}) %\cite{1983MN} 
that
\begin{eqnarray}\label{4.25}
	\|\nabla^2u\|_{L^2}^2+\|\nabla \pi\|_{L^2}^2\leq \|u_t\|_{L^2}^2+\|u\cdot \nabla u\|_{L^2}^2+\|\partial_1b\|_{L^2}^2+\|b\cdot\nabla b\|_{L^2}^2.
\end{eqnarray}
Further, we have
\begin{eqnarray}\label{4.26}
	\|\nabla^2u\|_{L^2}^2+\|\nabla \pi\|_{L^2}^2&\leq& \|u_t\|_{L^2}^2+\|u\cdot \nabla u\|_{L^2}^2+\|b\|_{H^1}^2+\|b\cdot\nabla b\|_{L^2}^2\nonumber\\
	&\leq&\|u_t\|_{L^2}^2+\|b\|_{H^1}^2+\|u\|_{L^4}^2\|\nabla u\|_{L^4}^2+\|b\|_{H^1}^2\|\nabla b\|_{H^1}^2\nonumber\\
	&\leq& \|u_t\|_{L^2}^2+\|b\|_{H^1}^2+\|u\|_{L^2}^2\|\nabla^2u\|_{L^2}^2+\|b\|_{H^1}^2\|\nabla b\|_{H^1}^2.
\end{eqnarray}
 Moreover, we have the $H^1$ estimate of $\nabla^2u$. By Lemma 2.1, we obtain
 \begin{eqnarray}\label{4.27}
 	\|\nabla^2u\|_{H^1}^2+\|\nabla \pi\|_{H^1}^2
 	&\leq& \|u_t\|_{H^1}^2+\|\partial_1b\|_{H^1}^2+\|u\cdot\nabla u\|_{H^1}^2+\|b\cdot\nabla b\|_{H^1}^2\nonumber\\
 	&\leq&\|(u_t,\partial_1 b)\|_{H^1}^2+\|u\nabla u\|_{L^2}^2+\||\nabla u|^2\|_{L^2}^2+\|u\nabla^2u\|_{L^2}^2\nonumber\\
 	&&+\|b^1\|_{L^\infty}^2\|\partial_1b\|_{H^1}^2+\|b^2\|_{L^\infty}^2\|\partial_2b\|_{H^1}^2\nonumber\\
 	&\leq&\|(u_t,\partial_1 b)\|_{H^1}^2+\|u\|_{H^2}^2\|\nabla^2 u\|_{H^1}^2
 	+\|b\|_{H^2}^2\|\partial_1b\|_{H^1}^2.
 \end{eqnarray}

 {\bf Step 8. Closing of the $a~priori$ estimates}\\
 We firstly add these estimates of {\bf Step 1-6} together to obtain that with $\frac{\delta_3}{2}>\delta_5>\delta_6>\delta_4>0$,
 \begin{eqnarray}\label{4.28}
 	&&\frac{d}{dt}\|(u,b)\|_{H^1}^2 +\delta_3\frac{d}{dt}\langle b,\partial_1u\rangle+\delta_4\frac{d}{dt}\Big(\langle \partial_2^2b^1,(\partial_2b^1)^2\rangle+\langle b^1,(\partial_2^2b^1)^2\rangle\Big)
 	+\delta_4\frac{d}{dt}\|\mathcal{A}b\|_{L^2}^2\nonumber\\
 	&&+\delta_6\frac{d}{dt}\|u_t\|_{L^2}^2
 	+\|\nabla u\|_{H^1}^2+\min\{\delta_3,\delta_4\}\|\partial_1b\|_{H^1}^2+\min\{\delta_5,\delta_6\}\|u_t\|_{H^1}^2+\delta_5\|b_t\|_{L^2}^2\nonumber\\
 	&\leq& C(\|(u,b)\|_{H^2}+\|(u,b)\|_{H^2}^2+\|b\|_{H^2}^{\frac52})
 	\times(\|\nabla u\|_{H^2}^2+\|\partial_1b\|_{H^1}^2+\|u_t\|_{H^1}^2+\|b_t\|_{L^2}^2)\nonumber\\
 	&&+\tilde{\epsilon} \|b\|_{H^2}^2.
 \end{eqnarray}
 There exist some suitable small $\delta_3>0$ and $\delta_4>0$ such that (\ref{4.28}) becomes 
 \begin{eqnarray}\label{4.29}
 	&&\frac{d}{dt}\Big(\|(u,b)\|_{H^1}^2 +\delta_4\|\mathcal{A}b\|_{L^2}^2+\delta_6\|u_t\|_{L^2}^2\Big)+\|\nabla u\|_{H^1}^2\nonumber\\
 	&&+\min\{\delta_3,\delta_4,\delta_5,\delta_6\}\Big(\|\partial_1b\|_{H^1}^2+\|u_t\|_{H^1}^2+\|b_t\|_{L^2}^2\Big)\nonumber\\
 &\leq& C(\|(u,b)\|_{H^2}+\|(u,b)\|_{H^2}^2+\|b\|_{H^2}^{\frac52})	\times(\|\nabla u\|_{H^2}^2+\|\partial_1b\|_{H^1}^2+\|u_t\|_{H^1}^2+\|b_t\|_{L^2}^2)\nonumber\\
 &&+\tilde{\epsilon}\|b\|_{H^2}^2.
 \end{eqnarray}
Taking $\delta=\min\{\delta_3,\delta_4,\delta_5,\delta_6\}$ and adding $\frac{\delta}{2}\times(\ref{4.27})$ together with (\ref{4.29}) yields
 \begin{eqnarray}\label{4.30}
	&&\frac{d}{dt}\Big(\|(u,b)\|_{H^1}^2 +\delta_4\|\mathcal{A}b\|_{L^2}^2+\delta_6\|u_t\|_{L^2}^2\Big)+\|\nabla u\|_{H^2}^2+\frac {\delta}{2}\Big(\|\partial_1b\|_{H^1}^2+\|u_t\|_{H^1}^2+\|b_t\|_{L^2}^2\Big)\nonumber\\
	&\leq& C(\|(u,b)\|_{H^2}+\|(u,b)\|_{H^2}^2+\|b\|_{H^2}^{\frac52})	\times(\|\nabla u\|_{H^2}^2+\|\partial_1b\|_{H^1}^2+\|u_t\|_{H^1}^2+\|b_t\|_{L^2}^2)\nonumber\\
	&&+\tilde{\epsilon}\|b\|_{H^2}^2.
\end{eqnarray}
 Integrating time to (\ref{4.30}) from 0 to $t$, we deduce that for $t\in (0,T]$  with the help of Lemma 2.2 (2),
 \begin{eqnarray}\label{4.31}
 	&&\|u\|_{H^1}^2+\|b\|_{H^2}^2+\|u_t\|_{L^2}^2
 	+C(\delta)\int_0^t\|\nabla u\|_{H^2}^2+\|\partial_1b\|_{H^1}^2+\|b_t\|_{L^2}^2+\|u_t\|_{H^1}^2d\tau\nonumber\\
 	&\leq& C(\|u_0\|_{H^1}^2+\|b_0\|_{H^2}^2+\|u_t(0)\|_{L^2}^2)+
 	C\int_0^t(\|(u,b)\|_{H^2}+\|(u,b)\|_{H^2}^2+\|b\|_{H^2}^{\frac52})\nonumber\\
 	&&	\times(\|\nabla u\|_{H^2}^2+\|\partial_1b\|_{H^1}^2+\|u_t\|_{H^1}^2+\|b_t\|_{L^2}^2)dt+C\tilde{\epsilon} T\sup_{t\in [0,T]}\|b(t)\|_{H^2}^2,
 \end{eqnarray} 
where we take $\tilde{\epsilon}$ satisfying  $\tilde{\epsilon}T\leq \frac{1}{2C}$.\\
Adding $\frac12 \times(\ref{4.26})$ together with (\ref{4.31}) together yields
 \begin{eqnarray}\label{4.32}
 && \|(u,b)\|_{H^2}^2+\|u_t\|_{L^2}^2+\|\nabla \pi\|_{L^2}^2+C(\delta)\int_0^t\|\nabla u\|_{H^2}^2+\|b_t\|_{L^2}^2+\|\partial_1b\|_{H^1}^2+\|u_t\|_{H^1}^2d\tau\nonumber\\
 	&\leq& C(\|(u_0,b_0)\|_{H^2}^2+\|u_t(0)\|_{L^2}^2)+	C\int_0^t(\|(u,b)\|_{H^2}+\|(u,b)\|_{H^2}^2+\|b\|_{H^2}^{\frac52})\nonumber\\
 	&&	\times(\|\nabla u\|_{H^2}^2+\|\partial_1b\|_{H^1}^2+\|u_t\|_{H^1}^2+\|b_t\|_{L^2}^2)dt.
 \end{eqnarray}
Finally, we are now in the position to prove the $a~priori$ estimate 
$$\mathcal{E}(t)\leq \frac{\epsilon^2}{2}$$ 
under the $a~priori$ assumption that
$$ \mathcal{E}(t)\leq \epsilon^2.$$ 
Set $\mathcal{E}(t)=\|(u,b)\|_{H^2}^2+\|(u_t,\nabla \pi)\|_{L^2}^2$ and $\mathcal{F}(t)=\int_0^t\|\nabla u\|_{H^2}^2+\|b_t\|_{L^2}^2+\|\partial_1b\|_{H^1}^2+\|u_t\|_{H^1}^2d\tau$.
When $\mathcal{E}(t)\leq \epsilon^2$, it follows from (\ref{4.32}) that
\begin{eqnarray}\label{4.33}
	\mathcal{E}(t)+C(\delta)\mathcal{F}(t)&\leq& C\mathcal{E}(0)+\Big(\mathcal{E}^{\frac12}+\mathcal{E}+\mathcal{E}^{\frac54}\Big)\mathcal{F}(t)\nonumber\\
	&\leq&C\mathcal{E}(0)+C\epsilon\mathcal{F}(t)\nonumber\\
	&\leq&C\mathcal{E}(0).
\end{eqnarray}
Hence, taking $\mathcal{E}(0)\leq \frac{\epsilon^2}{2C}$, we obtain 
$$\mathcal{E}(t)+C\mathcal{F}(t)\leq \frac{\epsilon^2}{2}.$$
This completes Proposition 3.1.
 
\end{proof}
Now, we are in position to give the proof of first part for existence in Theorem 2.1.

 Therefore,  with the suitable initial data, there exists a positive constant $\tilde{C}$ such that the initial total energy $\mathcal{E}(0)\leq \tilde{C}\epsilon$. According to the standard local well-posedness theory which can be obtained by classical arguments, there exists a positive $T$ such that for $C^*=C\tilde{C}$, 
 \begin{equation}\label{4.34}
 	\mathcal{E}(t)+\mathcal{F}(t)\leq C^*\epsilon,~~\forall t\in [0,T].
 \end{equation}
 Let $T^*$ be the largest possible time $T$ satisfying (\ref{4.34}), it is then to show $T^*=+\infty$. Indeed, we can use a standard continuation argument to get the desired result provided that estimate (\ref{4.33}) and $\epsilon$ is small enough.

 \vspace{6mm}
 
 \setcounter{equation}{0}
 \section{ Time asymptotic estimates}
  In this section, notice from (\ref{4.34}) that
  \begin{equation}\label{5.1}
 	\left\{
 	\begin{array}{l}
 	\int_0^\infty \|\nabla u\|_{H^2}^2dt\leq O(1),\\
 	\int_0^\infty \|u_t\|_{H^1}^2dt\leq O(1).
 	\end{array} \right.
 \end{equation} 
 It follows from (\ref{5.1}) that
 \begin{eqnarray}\label{5.2}
 		\int_0^\infty |\frac{d}{dt}\|\nabla u\|_{L^2}^2|dt\leq C\int_0^\infty \|\nabla u\|_{L^2}^2+\|\nabla u_t\|_{L^2}^2dt\leq O(1).
 \end{eqnarray}
 Combining (\ref{5.2}) with $(\ref{5.1})_1$ yields
 \begin{equation}\label{5.3}
 	\lim_{t\rightarrow \infty}\|\nabla u\|_{L^2(\R^2_+)}=0.
 \end{equation} 
 Since the Sobolev's inequality $$\|u\|_{L^p}\leq \|\nabla u\|_{L^2}^\theta\|u\|_{L^2}^{1-\theta}$$ for $2<p<\infty$ and $0<\theta<1$, we obtain
 \begin{equation}\label{5.4}
 \lim_{t\rightarrow \infty}	\|u\|_{L^p}=0,~~2<p<\infty.
 \end{equation}
We now aim to get the asymptotic estimate of magnetic field $b$.
For $2<p<\infty$, 
\begin{eqnarray}\label{5.5}
	\int_0^\infty\|b\|_{L^p}^qd\tau &\leq& \int_0^\infty \|b\|_{L^\infty}^{(1-\frac2p)q}\|b\|_{L^2}^{\frac2pq}d\tau\leq C\int_0^\infty \|b\|_{H^1}^{\frac12(1-\frac2p)q}\|\partial_1b\|_{H^1}^{\frac12(1-\frac2p)q}d\tau \nonumber\\
	&\leq& O(1)\int_0^\infty \|\partial_1b\|_{H^1}^\sigma d\tau\leq O(1),
\end{eqnarray}
where $q\geq \frac{4p}{p-2}$ or $\sigma\geq 2$.
In addition, 
\begin{eqnarray}\label{5.6}
	\int_0^\infty |\frac{d}{dt}\|b\|_{L^p}^q|dt&\leq& O(1)\int_0^\infty \|b\|_{L^p}^{q-p}\|\partial_tb\|_{L^2}\|b\|_{L^{2(p-1)}}^{p-1}dt\nonumber\\
	&\leq& O(1)\int_0^\infty \|\partial_tb\|_{L^2}^2+\|b\|_{L^p}^{2(q-p)}\|b\|_{2(p-1)}^{2(p-1)}dt\nonumber\\
	&\leq& O(1)\int_0^\infty \|\partial_tb\|_{L^2}^2+\|b\|_{L^p}^qdt\nonumber\\
	&\leq& O(1),
\end{eqnarray}
where $q\geq 2p$.
Hence, 
%$$ 
%	\left\{
%	\begin{array}{l}
%		p>2,\\
%		q\geq \frac{4p}{p-2},\\
%		q\geq 2p,
%	\end{array} \right.
%$$
it follows from (\ref{5.5}) and (\ref{5.6})  that for $ 2< p<\infty$, 
\begin{eqnarray*}
	\|b\|_{L^p}\rightarrow0, ~~t\rightarrow\infty.
\end{eqnarray*}
Therefore, that completes the proof of Theorem 2.1.

%%%%%%%%%%%%%%%%%%%%%%%%%%%%%%%

\newpage
\bibliography{reference}

\end{document}